\documentclass[11pt,reqno]{amsart}
\usepackage[T1]{fontenc}
\usepackage{a4, amsmath}
\usepackage{mathtools}
\usepackage{amssymb}
\usepackage{amsthm, amscd, bm, mathdots}
\usepackage{enumerate}
\usepackage[inline]{enumitem}
\usepackage{tabto}
\usepackage{hyperref}
\usepackage{cleveref}
\usepackage{datetime}
\usepackage{xcolor}

\usepackage[paper=a4paper,top=1in, bottom=1in, left=1in, right=1in]{geometry}

\theoremstyle{plain}
\newtheorem{thm}{Theorem}[section]
\newtheorem{cor}[thm]{Corollary}
\newtheorem{lem}[thm]{Lemma}
\newtheorem{prp}[thm]{Proposition}

\theoremstyle{remark}
\newtheorem{expl}[thm]{Example}

\numberwithin{equation}{section}

\crefname{thm}{theorem}{theorems}
\crefname{lem}{lemma}{lemmata}
\crefname{prp}{proposition}{propositions}
\crefname{cor}{corollary}{corollaries}

\newcommand{\mc}{\mathcal}

\newcommand{\N}{\mathbb N}
\newcommand{\Z}{\mathbb Z}
\newcommand{\Q}{\mathbb Q}
\newcommand{\R}{\mathbb R}

\newcommand{\F}{\mathbb F}
\newcommand{\co}{\mathcal O}

\DeclareMathOperator{\cls}{cls}

\DeclareMathOperator{\Tr}{Tr}
\DeclareMathOperator{\Gal}{Gal}
\DeclareMathOperator{\NN}{N}

\DeclareMathOperator{\sgn}{sgn}
\DeclareMathOperator{\sign}{sign}
\DeclareMathOperator{\sgnrk}{sgnrk}

\DeclareMathOperator{\Conv}{Conv}
\DeclareMathOperator{\ID}{ID}
\DeclareMathOperator{\IV}{IV}
\DeclareMathOperator{\Span}{Span}
\DeclareMathOperator{\Vol}{Vol}

\newcommand{\bea}{\begin{enumerate}[label=(\alph*)]}
\newcommand{\ben}{\begin{enumerate}[label=(\arabic*)]} 
\newcommand{\ber}{\begin{enumerate}[label=(\roman*)]}
\newcommand{\ee}{\end{enumerate}} 

\newcommand{\cb}[1]{\left\{{#1}\right\}}
\newcommand{\cbm}[2]{\left\{{#1}\;\middle|\;{#2}\right\}}
\newcommand{\flr}[1]{\left\lfloor{#1}\right\rfloor}
\newcommand{\ol}[1]{\overline{#1}}
\newcommand{\pb}[1]{\left\langle{#1}\right\rangle}
\newcommand{\pbs}[1]{\langle{#1}\rangle}
\newcommand{\rb}[1]{\left({#1}\right)}
\newcommand{\sqb}[1]{\left[{#1}\right]}

\title{Sails for universal quadratic forms}
\author{Vítězslav Kala}
\author{Siu Hang Man}
\address{Charles University, Faculty of Mathematics and Physics, Department of Algebra, Sokolov\-ská 83, 186~75 Praha~8, Czech Republic}
\email[V.~Kala]{vitezslav.kala@matfyz.cuni.cz}
\email[S.~H.~Man]{shman@karlin.mff.cuni.cz}

\thanks{The research was supported by {Czech Science Foundation} grant 21-00420M (Kala, Man), and by {Charles University} programme PRIMUS/24/SCI/010 (Man).}

\begin{document}
\begin{abstract}
We establish a new connection between sails, a key notion in the geometric theory of generalised continued fractions, and arithmetic of totally real number fields, specifically, universal quadratic forms and additively indecomposable integers. 
Our main application is to biquadratic fields, for which we show that if their signature rank is at least 3, then ranks of universal forms and numbers of indecomposables grow as a power of the discriminant. We also construct a family in which these numbers grow only logarithmically.
\end{abstract}

\makeatletter
\@namedef{subjclassname@2020}{%
  \textup{2020} Mathematics Subject Classification}
\makeatother

\maketitle

\section{Introduction}

Quadratic forms and the integers represented by them have been influential from the dawn of number theory till today. 
Their study used and often initiated the development of a broad range of tools, from elementary and combinatorial (a modern example being escalations used in the proof of 290-Theorem), through algebraic (such as the classical connection to class groups of real quadratic number fields or local-global principles) to analytic (theta functions, circle method, or Siegel's mass formula). In this paper, we aim to initiate the use of \textit{sails} in the study of universal quadratic forms. However, let us discuss some background before turning to sails and our new connection.

\medskip

Over rational integers, 290-Theorem of Bhargava--Hanke \cite{BH} completely characterised positive definite quadratic forms that are \textit{universal} in the sense that they represent all positive integers: Their universality is equivalent to the representability of all the integers $1,2,3,\dots, 290$.

The situation over rings of integers $\co_K$ in number fields $K$ is much more complicated.
We call a quadratic form $Q$ over $\co_K$ \textit{totally positive definite} if all its values are totally positive, except for $Q(0,\dots,0)=0$, and \textit{universal} if it moreover represents all elements of $\co_K^+$, i.e., all totally positive integers.
In the 1940's, Siegel \cite{Si2} established that $\Q$ and $\Q(\sqrt 5)$ are the only totally real number field over which the sum of any number of squares is universal; four (Lagrange, 1770) and three (Maa{\ss} \cite{Ma}) squares  suffice, respectively.
This was the first suggestion that universal forms over totally real number fields may be quite scarce.
In contrast, when the field is not totally real, then the sum of squares is universal (in the sense that it represents all elements with no negative real conjugates) if and only if the field has odd discriminant \cite{Si2}. More generally, when the quadratic form is not totally positive definite, then local methods give much more information, e.g., see \cite{EH2, HHX}.

While the asymptotic local-global principle of Hsia--Kitaoka--Kneser \cite{HKK1978} implies that universal quadratic forms exist over every (totally real) number field, it does not provide any information about their ranks. Still unproven is even the influential Kitaoka's conjecture from the 1990's that there are only finitely many totally real number fields with a ternary universal quadratic form, despite interesting results for quadratic \cite{CKR} and biquadratic \cite{KTZ2020} fields, for fields of odd degree \cite{EK1}, and a weaker statement in every fixed degree \cite{KY3}. Further, Kim--Kim--Park \cite{Ki2,KKP} studied real quadratic fields with universal forms of ranks up to eight.

Starting in 2015, Blomer--Kala \cite{BK1, BK2,Ka1} pioneered the use of two new tools in the area: additively indecomposable integers and continued fractions. An element $\alpha\in\co_K^+$ is \textit{indecomposable} if it does not decompose as $\alpha=\beta+\gamma$ for any $\beta,\gamma\in\co_K^+$. 
As already observed by Siegel \cite{Si2}, such elements are hard to represent by a quadratic form, and in real quadratic fields, Dress--Scharlau \cite{DS} showed  that they precisely correspond to semiconvergents of periodic continued fractions.
Together, these notions can be used to establish the surprising result \cite{BK1} that, for every $r$, there are real quadratic fields without a universal quadratic form of rank at most $r$.
Along with the study of small trace elements in the codifferent \cite{KT2023,Ya}, Kala--Yatsyna--Żmija \cite{KYZp2023} recently even showed that, when ordered by discriminant, density 1 of real quadratic fields do not admit a universal quadratic form of rank at most $r$.

\medskip

Also over number fields of higher degrees, the importance of indecomposables was quite clear, but the absence of suitable generalised continued fractions posed a significant obstacle, in spite of some interesting partial results \cite{Ka4,KT2023,KL,KTZ2020,Man,Ya}.

To understand the difficulty, note that while there are various  generalizations of continued fractions, in particular multidimensional ones such as the Jacobi--Perron algorithm (JPA,  \cite{jacobi, Perron}), they are still not well understood. In particular, a major open question is Hermite's problem from 1839 that asks about
the existence of periodic expansions for algebraic numbers of degree $>2$. This is still wide open, even in the cubic case, where computational evidence suggests that the JPA expansion of $\sqrt[3]n$ is typically not periodic \cite{Voutier}. However, Karpenkov \cite{Karp} recently constructed a modification of JPA that has periodic expansions of totally real cubic vectors. Nevertheless, these multidimensional continued fraction algorithms are still quite mysterious, making it difficult to use for our purposes (see \cite{KST}).

In this paper, we break through this barrier by connecting universal forms and indecomposables to the \textit{geometric approach to generalised continued fractions}, first considered by Klein in 1895 \cite{Klein}, and recently revived by Arnold and his colleagues, including Karpenkov \cite{Karpenkov2022}.

Let $K$ be a totally real number field of degree $n$ with ring of integers $\mc O_K$ and unit group $\mc O_K^\times$.
The Minkowski embedding $i:K\hookrightarrow \R^n$ is then defined by the $n$ embeddings $\tau_i: K\hookrightarrow \R$ and the image $\Lambda := i(\mc O_K) \subseteq \R^n$ is a lattice of rank $n$. As we are interested in totally positive elements, we consider the set of lattice points $\Lambda^+ := \Lambda \cap \R_+^n$ in the positive octant. The \emph{Klein polyhedron} is the convex hull $\mc K_K := \Conv(\Lambda^+)$, and its boundary $\mc S_K := \partial\mc K_K$ is called the \emph{sail} associated to $\Lambda^+$ (or $\mc O_K^+$, or just $K$). 

Informally, the lattice points lying on the sail are the points closest to the coordinate hyperplanes, in a suitable sense, and one may thus suspect that they may be related to indecomposables. We indeed show in Theorem \ref{thm:si} that this is the case: $\alpha$ is indecomposable if it lies on the sail, i.e., $i(\alpha)\in \mc S_K$. However, the other implication holds in the quadratic case $n=2$, but not in general.

Nevertheless, we propound that understanding sails is critical for the study of indecomposables (and, in turn, universal forms) and that, conversely, indecomposables provide a natural refinement of sails that should be of interest also for the continued fraction community.

\medskip

We start this paper by describing the face structure of a sail $\mc S_K$. Thanks to the action by multiplication by totally positive units, there is a finite set of faces $A_i$ whose translates give the whole sail, see \eqref{eq:decomp}; topologically, these faces $A_i$ give a \textit{torus decomposition} $[0,1]^n=\bigcup_i A_i$.
As the corresponding cones cover the totally positive octant \eqref{eq:shint}, each indecomposable correspondingly lies ``above'' one of the faces $A_i$. Suitable counting of lattice points then allows us to estimate the total number of indecomposables (Theorem \ref{thm:ii}). This is particularly nice and useful for cubic fields where Theorem \ref{thm:cii} gives an equality that allows us to significantly simplify previous arguments (due to Kala--Tinková \cite{KT2023, Ti1}) determining indecomposables in families such as Shanks' simplest cubic fields \cite{Shanks1974}.

\medskip

For  $\alpha \in K^\times$, we define its \emph{signature} $\sgn(\alpha) \in \{+,-\}^n$ as the $n$-tuple of signs of $\tau_i(\alpha)$. It is often convenient to view $\sgn(\alpha)$ as a vector in the \emph{signature space} $\F_2^n$.
The collection of signatures $\sgn(\alpha)$ for $\alpha \in \mc O_K^\times$ is the \emph{unit signature group} of $K$, and its dimension (as a subspace of $\F_2^n$) is called the \emph{unit signature rank} of $K$, denoted by $\sgnrk(K)$. As $-1\in \mc O_K^\times$, we have $1\leq \sgnrk(K)\leq n$.
Available unit signatures control symmetries between various octants of the Minkowski space, and so have key influence on the structure of the sail, as well as indecomposables and universal forms. For example, Kitaoka's conjecture (for classical quadratic forms) is non-trivial only when $\sgnrk(K)=n-1,n$.

In our paper, we then turn to the case of biquadratic fields, i.e., extensions $K=\Q(\sqrt {D_1}, \sqrt {D_2})$ of degree $n=4$ (with squarefree $D_1,D_2\in\Z_{>1}$). Their great advantage is that they contain the quadratic subfields $\Q(\sqrt {D_1}), \Q(\sqrt {D_2}), \Q(\sqrt {D_1D_2})$. While the indecomposables for these subfields need not remain indecomposable in $K$ (although they often do, see \cite{CL+, KTZ2020, Man}), they still lie close to the sail (as they have small traces). Crucially for us, quadratic fields $F$ always contain $\gg \Delta_F^{1/2}$ indecomposables in the quadrant of signature $(+,-)$, and,
when the biquadratic field has $\sgnrk(K)=3,4$, there are enough units to make some of these indecomposables totally positive in $K$. 

These considerations imply the following theorem, for which we denote by $R_{\cls}(K)$ the minimal rank of a universal quadratic lattice that is classical (see \Cref{sect:lattices} below; similarly $R(K)$ denotes the minimal rank without the classical assumption), and by $\iota(K)$ the number of indecomposables modulo multiplication by totally positive units.

\begin{thm}\label{thm:usr}
Let $K$ be a totally real biquadratic field with $\sgnrk(K) \ge 3$. Then we have $R_{\cls}(K), \iota(K) \gg \Delta_K^{1/12}$ and $R(K) \gg \Delta_K^{1/42}$ for some explicitly computable constants, where $\Delta_K$ denotes the discriminant of $K$.
\end{thm}

Here we use the usual analytic notation where $f\gg g$ (or $g\ll f$) denotes that there is a constant $C>0$ such that $f(x)>Cg(x)$ for all admissible values of $x$, and $f\asymp g$ if $f\gg g$ and $f\ll g$.

This theorem will be proved in Subsection \ref{subsec:pf 1.1}. Perhaps curiously, it covers exactly the hard cases for Kitaoka's conjecture, and in fact, it can be used to simplify a part of its proof for biquadratic fields (that is due to Krásenský--Tinková--Zemková \cite{KTZ2020}).

Theorem \ref{thm:usr} leaves open the enticing question of the behaviour of $R(K), \iota(K)$ in the case $\sgnrk(K) \leq 2$.
As we have seen in \cite{Man}, $\iota(K)$ usually grows with the discriminant of $K$, it is hard to find examples of biquadratic fields with small value of $\iota(K)$, and it even seems that there may be only finitely many biquadratic fields with $\iota(K)\leq R$ for each $R$ \cite[Conjecture 1.7]{Man}.
Using our improved understanding motivated by the geometry of sails, we construct a family where the number of indecomposables grows only logarithmically, while the best examples known so far involved $\iota(K) \asymp \Delta_K^{1/8}$
\cite[Section 5]{Man}).

\begin{thm}\label{thm:pfi simplified} 
	Let $p_n$ be defined as in \eqref{eq:xy4} and let $K =  \Q(\sqrt{5}, \sqrt{p_n})$. Then we have
	\[
	\frac{\log \Delta_K}{8w} \le \iota(K) \le \frac{\log \Delta_K}{8w} + 1,
	\]
	where $w := \log (\frac{1+\sqrt{5}}2)$. 
\end{thm}

A more precise version of this result will be proved as Theorem \ref{thm:pfi}.
Note that, since squarefree integers have positive density in $\N$, we expect that \Cref{thm:pfi simplified} actually describes an infinite family of biquadratic fields. However, showing that such an exponentially sparse sequence takes infinitely many squarefree values is far out of reach with current techniques.
By the relation between $R(K)$ and $ \iota(K)$ (see \Cref{cor:rfif}), Theorem \ref{thm:pfi simplified} also implies that $R(K)\leq R_{\cls}(K) \ll \log \Delta_K$.

\medskip

Let us conclude by pointing out some possible directions of future research (that we are already investigating).
First of all, one can consider sails in all signatures, not only in the totally positive octant, as they carry additional information about the number field $K$. Units that are not totally positive give isomorphisms between some of these sails
(this was already the idea behind the proof of \Cref{thm:usr}!), and so the other sails are useful especially when $K$ do not have all unit signatures, i.e., $\sgnrk (K)<[K:\Q]$.
However, a natural question is how are the different sails related; can they be (partially) reconstructed from the ``totally positive'' sail?

\section*{Acknowledgments}
We thank Oleg Karpenkov for kindly answering our questions, Pavlo Yatsyna for interesting discussions, and the anonymous referee for very helpful comments and corrections.

\section{Sails for totally real number fields}

Let $K$ be a totally real number field of degree $n$. Then we have $n$ distinct real embeddings $\tau_1,\ldots,\tau_n : K\hookrightarrow \R$. Given an element $\alpha \in K^\times$, we define the \emph{signature} of $\alpha$ to be 
\[
\sgn(\alpha) := \rb{\sign(\tau_1(\alpha)),\ldots,\sign(\tau_n(\alpha))},
\]
where $\sign(x)$ denotes the sign of the real number $x\in\R$. Often we view $\sign(x)$ as an element lying in the additive group $\F_2$, so $\sign(x)$ has value $0$ if $x>0$ and value $1$ if $x<0$, and $\sgn(\alpha)$ is viewed as a vector in the \emph{signature space} $\F_2^n$. Let $\mc O_K$ denote the ring of integers of $K$, and $\mc O_K^\times$ the group of units in $\mc O_K$. The collection of signatures $\sgn(\alpha)$ for $\alpha \in \mc O_K^\times$ is called the \emph{unit signature group} of $K$, and the rank of this subspace is called the \emph{unit signature rank} of $K$, denoted by $\sgnrk(K)$.

We have the Minkowski embedding $i:K\hookrightarrow \R^n$, defined by the $n$ embeddings $\tau_i: K\hookrightarrow \R$. Then the image $\Lambda := i(\mc O_K) \subseteq \R^n$ is a lattice of rank $n$ whose covolume is $\Vol(\Lambda) = \Delta_K^{1/2}$. By a mild abuse of notation, we shall often refer to a lattice point $v\in\Lambda$ by its corresponding element $\alpha \in \mc O_K$.

\medskip

Let us further recall some notions from integer geometry relevant for our purposes. A detailed discussion on the subject can be found in \cite{Karpenkov2022}.

An \emph{integer plane} is an affine subspace $A \subseteq \R^n$ such that the sublattice $A\cap \Lambda$ has the same rank as the dimension of $A$. We say that a (convex) polytope, not necessarily of full dimension is \emph{integer} if all of its vertices lie in $\Lambda$. The \emph{integer volume} of an integer simplex $S = \pb{v_0,\ldots,v_d}$, denoted by $\IV(S)$, is the index of the sublattice generated by the vectors $v_1-v_0,\ldots,v_d-v_0$ in the lattice $\Lambda_S := \Lambda \cap \Span_\R(v_1-v_0,\ldots,v_d-v_0)$. The integer volume is related to the usual Euclidean volume via the formula
\[
\Vol(S) = \frac{1}{d!}\IV(S)\Vol(\Lambda_S),
\]
where $d$ is the dimension of the simplex $S$. Since an integer polytope can be decomposed into a disjoint union of integer simplices of the same dimension, we may define the integer volume of an integer polytope to be the sum of the integer volumes of the constituent integer simplices; the integer volume is independent of the decomposition. 

Let $A \subseteq \R^n$ be an integer plane which does not intersect with the origin. The \emph{integer distance} of $A$ from the origin, denoted by $\ID(A)$, is the index of the sublattice generated by $A\cap\Lambda$ in the integer lattice $\Lambda \cap \Span_\R(A)$. By convention, if $A$ intersects with the origin, then we say the integer distance of $A$ from the origin is $0$. We shall also define the integer distance $\ID(S)$ of a polytope $S$ from the origin to be $\ID(A)$, where $A$ is the integer plane of the same dimension containing $S$. 

Let $\Tr_{K/\Q}: K \to \Q$, $\alpha \mapsto \sum_{i=1}^n \tau_i(\alpha)$, be the \textit{trace} function for the field extension $K/\Q$, and $\mc O_K^\vee := \cbm{\delta\in K}{\Tr_{K/\Q}(\delta\alpha) \in \Z \ \forall\alpha\in\mc O_K}$ the \textit{codifferent}. In the article we often consider hyperplanes of the form
\[
\cbm{\alpha\in K}{\Tr_{K/\Q}(\delta\alpha) = k} \subseteq K \simeq \Q^n, \quad \delta\in K^\times, \quad k\in\Q.
\]
We often denote such a hyperplane (and its canonical completion in the Minkowski space $\R^n$) by the equation $\Tr_{K/\Q}(\delta(-)) = k$.

\begin{lem}\label{lem:ct} 
Let $A \subseteq \R^n$ be an integer plane with integer distance $\ID(A) = k > 0$. Then there exists $\delta \in \mc O_K^\vee$ such that $A$ is contained in the hyperplane $\Tr_{K/\Q}(\delta(-)) = k$. 
\end{lem}
\begin{proof}
By extending $A$, we may assume $A$ has dimension $n-1$. Let $\alpha_0,\ldots, \alpha_{n-1} \in \mc O_K$ correspond to a set of generators of $A$. By picking a $\Q$-basis of $K$, we obtain a system of linear equations $\Tr_{K/\Q}(\delta \alpha_i) = k$ over $\Q$. It follows that there exists $\delta \in K$ satisfying $\Tr_{K/\Q}(\delta\alpha_i) = k$. Now let $q \in \Q$ be the minimal positive value of $\Tr_{K/\Q}(\delta\alpha)$ for $\alpha\in\mc O_K$. Then $\ID(A) = k/q = k$, which implies $q=1$. Thus $\delta$ lies in the codifferent $\mc O_K^\vee$.
\end{proof}

The argument above also works backwards, giving us the following corollary.
\begin{cor}\label{cor:ctc} 
Suppose an integer plane $A$ is contained in the hyperplane $\Tr_{K/\Q}(\delta(-)) = k$ for some $\delta \in \mc O_K^\vee$ and $k>0$. Then $\ID(A)$ is a divisor of $k$. In particular, if $A$ is contained in the hyperplane $\Tr_{K/\Q}(\delta(-)) = 1$ for some $\mc O_K^\vee$, then $\ID(A) = 1$.
\end{cor}

Now we consider the set of lattice points $\Lambda^+ := \Lambda \cap \R_+^n$ in the positive octant, which corresponds to the set of totally positive integers $\mc O_K^+$. The \emph{Klein polyhedron} is the convex hull $\mc K_K := \Conv(\Lambda^+)$. The boundary $\mc S_K := \partial\mc K_K$ is called the \emph{sail} associated to $\Lambda^+$ (or $\mc O_K^+$). The Klein polyhedron is invariant under multiplication by totally positive units, so the sail $\mc S_K$ also has a periodic structure, and we can find a finite set of faces $A_i$, such that $\mc S_K$ is a union of their translates by totally positive units (see \cite[Ch. 22]{Karpenkov2022}):
\begin{equation}\label{eq:decomp}
\mc S_K = \bigcup_{\epsilon \in \mc O_K^{\times,+}} \bigcup_i A_i\epsilon.
\end{equation}

It turns out that a good understanding of the sail $\mc S_K$ can yield useful information about the indecomposable elements, at least when the degree of the number field $K$ is small.

\begin{expl}
For real quadratic fields $K$, the structure of the indecomposable elements is fully understood in terms of continued fractions (see \Cref{sect:cf}). The sail $\mc S_K$ consists of segments connecting two consecutive upper convergents (and their conjugates). Assuming the notation in \Cref{sect:cf}, we note that the segment $B_i$ connecting the upper convergents $\beta_{2i-1}$ and $\beta_{2i+1}$ is contained in the line $\Tr_{K/\Q}(\delta_{2i}(-)) = 1$. By \Cref{cor:ctc}, we conclude that $\ID(B_i) = 1$ for all $i$. Since the upper semiconvergents $\beta_{2i-1,l}$ all lie on the segment $B_i$, and that the indecomposable elements of $\mc O_K^+$ are precisely the upper semiconvergents, we see that a totally positive integer $\alpha\in\mc O_K^+$ is indecomposable if and only if it lies on the sail $\mc S_K$. 
\end{expl}

In the general case, we show below in \Cref{thm:si} that the integers lying on the sail are still indecomposable. However, the converse is not quite true, and there are often indecomposable elements lying in the interior of the Klein polyhedron $\mc K_K$. For an example, we shall see in \Cref{expl:scf} that for the simplest cubic fields considered in \cite{KT2023}, there is an indecomposable element, unique up to multiplication by totally positive units, that sits in the interior of the Klein polyhedron.

\begin{thm}\label{thm:si} 
Let $K$ be a totally real number field. If $\alpha \in \mc O_K^+$ lies on the sail $\mc S_K$, then $\alpha$ is indecomposable. 
\end{thm}

To prove \Cref{thm:si}, we need a small lemma.

\begin{lem}\label{lem:tpc} 
Let $A_i$ be a face of $\mc S_K$. Then $A_i$ is contained in the hyperplane $\Tr_{K/\Q}(\delta(-)) = k$ for some $\delta \in \mc O_K^{\vee,+}$ and $k\in\N$.
\end{lem}
\begin{proof}
By \Cref{lem:ct}, we know that $A_i$ is contained in the hyperplane $\Tr_{K/\Q}(\delta(-)) = k$ for some $\delta\in\mc O_K^\vee$ and $k\in\N$. It remains to show that $\delta$ is totally positive. Suppose to the contrary that $\delta$ is not totally positive. Then the hyperplane $\Tr_{K/\Q}(\delta(-)) = 0$ intersects nontrivially with the positive octant $\R_+^n$, and in particular contains a totally positive element $\alpha \in \mc O_K^+$. On the other hand, since $A_i$ is a face of the sail $\mc S_K$, by convexity of the Klein polyhedron $\mc K_K$ we have $\Tr_{K/\Q}(\delta\alpha) \ge k$ for all $\alpha \in \mc O_K^+$, a contradiction. So $\delta \in \mc O_K^{\vee,+}$ is totally positive.
\end{proof}

\begin{proof}[Proof of Theorem $\ref{thm:si}$]
Suppose $\alpha \in \mc O_K^+$ lies on a face $A_i$ of the sail $\mc S_K$. By \Cref{lem:tpc}, $A_i$ is contained in the hyperplane $\Tr_{K/\Q}(\delta(-)) = k$ for some $\delta\in \mc O_K^{\vee,+}$ and $k\in\N$. In particular, we have $\Tr_{K/\Q}(\delta\alpha) = k$. Suppose $\alpha$ admits a decomposition $\alpha = \alpha_1+\alpha_2$, with $\alpha_1,\alpha_2 \in \mc O_K^+$. Then $\Tr_{K/\Q}(\delta\alpha_1)$ and $\Tr_{K/\Q}(\delta\alpha_2)$ are both positive, and we have
\[
\Tr_{K/\Q}(\delta\alpha_1) + \Tr_{K/\Q}(\delta\alpha_2) = \Tr_{K/\Q}(\delta\alpha) = k.
\]
Hence $\Tr_{K/\Q}(\delta\alpha_1), \Tr_{K/\Q}(\delta\alpha_2) < k$. On the other hand, as in the proof of \Cref{lem:tpc}, convexity of the Klein polyhedron $\mc K_K$ implies that $\Tr_{K/\Q}(\delta\gamma) \ge k$ for all $\gamma \in \mc O_K^+$, a contradiction. So $\alpha$ is indecomposable.
\end{proof}

To a face $A_i$ of the sail $\mc S_K$, we may associate a cone $\R_+ A_i$. By Shintani's unit theorem \cite{Shintani1976}, the translates of these cones by totally positive units cover the positive octant:
\begin{equation}\label{eq:shint}
	\R_+^n = \bigcup_{\epsilon\in\mc O_K^{\times,+}} \bigcup_i \R_+ A_i\epsilon.
\end{equation}
In particular, every indecomposable element in $\mc O_K^+$, up to translation by a totally positive unit, sits in some cone $\R_+ A_i$. So it suffices to consider the indecomposable elements in these cones. Let $A_i = \bigcup_j A_{i,j}$ be a decomposition of the face $A_i$ into integer simplices. This induces a decomposition $\R_+ A_i = \bigcup_j \R_+ A_{i,j}$ of the cone $\R_+ A_i$ into simplicial cones. Let $\alpha_1,\ldots, \alpha_n \in \mc O_K^+$ be the vertices of the integer simplex $A_{i,j}$. Then every indecomposable element in the cone $\R_+ A_{i,j}$ is actually contained in the parallelepiped
\[
P_{i,j} = \cbm{\sum_{k=1}^n \lambda_k \alpha_k}{0\le \lambda_k < 1, \; 1\le k \le n}.
\]
Noting that $P_{i,j}$ is the fundamental domain of the lattice generated by $\alpha_1,\ldots, \alpha_n$, it follows that 
\[
\#(\mc O_K^+ \cap P_{i,j}) = \IV(A'_{i,j}) - 1 = \IV(A_{i,j})\ID(A_{i,j}) - 1,
\]
where $A'_{i,j}$ is the simplex with vertices $0, \alpha_1,\ldots, \alpha_n$. In particular, this says that the number of indecomposable elements in the cone $\R_+ A_{i,j}$ which are not on $A_{i,j}$ is bounded above by $\IV(A_{i,j})\ID(A_{i,j}) - 1$. Denoting by $\iota_{\textrm{int}} (A_i)$ the number of indecomposable elements in the cone $\R_+ A_i$ which are not on the face $A_i$, we conclude the following theorem.
\begin{thm}\label{thm:ii} 
Let $A_i$ be a face of the sail $\mc S_K$, and $A_i = \bigcup_{j=1}^k A_{i,j}$ a decomposition of $A_i$ into $k$ simplices. Then $\iota_{\textrm{int}}(A_i) \le \ID(A_i)\IV(A_i)-k$. In particular, if $\ID(A_i) = 1$, and $A_i$ admits a unimodular triangulation (i.e., a decomposition into simplices with unit integer volume), then $\iota_{\textrm{int}}(A_i) = 0$, so all indecomposable elements in the cone $\R_+ A_i$ lie on the face $A_i$.
\end{thm}

On the other hand, the totally positive integers lying in the parallelepipeds $P_{i,j}$ need not be indecomposable, because it could possibly be written as a sum of two other totally positive integers inside (or outside) of the parallelepiped. However, we shall show that when $K$ is a cubic field, then the totally positive integers in $P_{i,j}$ are indeed indecomposable. This gives an explicit formula for $\iota_{\textrm{int}}(A_i)$.

\begin{thm}\label{thm:cii} 
Suppose $[K:\Q] = 3$, and let $A_i$ be a face of the sail $\mc S_K$. Then we have 
\[
\iota_{\textrm{int}}(A_i) = (\ID(A_i)-1)\IV(A_i).
\]
\end{thm}
\begin{proof}
We observe that if $[K:\Q] = 3$, then $A_i$ is $2$-dimensional, and always admits a unimodular triangulation $A_i = \bigcup_{j=1}^k A_{i,j}$ (see \cite[Proposition 2.17]{Karpenkov2022}). By \Cref{lem:ct}, we find $\delta \in \mc O_K^\vee$ such that $A_i$ is contained in the plane $\Tr_{K/\Q}(\delta(-)) = \ID(A_i)$. By the construction of the sail $\mc S_K$, we deduce that $\Tr_{K/\Q}(\delta \alpha) \ge \ID(A_i)$ for every $\alpha\in\mc O_K^+$. 

Now consider the parallelepiped $P_{i,j}$ associated to the simplex $A_{i,j} = \pb{\alpha_1,\alpha_2,\alpha_3}$. Let $\gamma \in \mc O_K^+ \cap P_{i,j}$. Then the element $\gamma' := \alpha_1+\alpha_2+\alpha_3 - \gamma$ is contained in the closure $\ol{P_{i,j}}$, and we have 
\[
\Tr_{K/\Q}(\delta\gamma) + \Tr_{K/\Q}(\delta\gamma') = \Tr_{K/\Q}(\delta(\alpha_1+\alpha_2+\alpha_3)) = 3\ID(A_i). 
\]
Since we know that $\Tr_{K/\Q}(\delta\alpha) \ge \ID(A_i)$ for $\alpha\in\mc O_K^+$, we immediately deduce that \[
\ID(A_i) \le \Tr_{K/\Q}(\delta\gamma)\le 2\ID(A_i).
\]
The inequality above can be made strict. To see this, we observe that $\alpha_1,\alpha_2,\alpha_3$ are the only lattice points on $A_{i,j}$, which is the intersection of $\ol{P_{i,j}}$ and the plane $\Tr_{K/\Q}(\delta(-)) = \ID(A_i)$. Since $\gamma$ and $\gamma'$ are not a sum of $\alpha_i$'s by construction, it follows that $\Tr_{K/\Q}(\delta\gamma), \Tr_{K/\Q}(\delta\gamma') \ne \ID(A_i)$. So we conclude a strict inequality
\[
\ID(A_i) < \Tr_{K/\Q}(\delta\gamma) < 2\ID(A_i)
\]
as claimed. Suppose $\gamma$ admits a decomposition $\gamma = \gamma_1 + \gamma_2$ as a sum of totally positive integers. Then one of $\Tr_{K/\Q}(\delta\gamma_1)$ and $\Tr_{K/\Q}(\delta\gamma_2)$ must be smaller than $\ID(A_i)$, a contradiction. So $\gamma$ is indecomposable. Since there are $\ID(A_i)-1$ such indecomposable elements in $P_{i,j}$, summing over $j$ yields the statements for cubic fields.
\end{proof}

Using \Cref{thm:cii} and the knowledge of the sail $\mc S_K$ of a totally real cubic field $K$, we may compute $\iota(K)$ without the need to completely characterise the indecomposable elements in $\mc O_K^+$, a process which is often difficult or tedious. We demonstrate this with the following example.
\begin{expl}\label{expl:scf} 
We consider Shanks' family of the simplest cubic fields \cite{Shanks1974}, defined as $K = \Q(\rho)$, where $\rho$ is the largest root of the polynomial
\[
x^3 - ax^2 - (a+3)x - 1,
\]
where $a \ge -1$ is an integer. Suppose also that $\mc O_K = \Z[\rho]$; it is well-known  that there are infinitely many such fields (see \cite[Section 3]{Cusick1984}, or \cite[Section 2]{KT2023}). Then $1, \rho, \rho^2$ is an integral basis of $\mc O_K$. The group $\mc O_K^{\times,+}$ of totally positive units is generated by $\varepsilon_1 = \rho^2$, and $\varepsilon_2 = 1+2\rho+\rho^2$. From \cite[Chapter 22.6]{Karpenkov2022}, we see that the sail consists of translates of the following two faces by totally positive units:
\[
A_1 = \pb{1,\varepsilon_1,\varepsilon_2}, \quad A_2 = \pb{1,\varepsilon_1,\varepsilon_1\varepsilon_2^{-1}}. 
\]
In terms of the basis $1,\rho,\rho^2$, we compute 
\[
\varepsilon_1\varepsilon_2^{-1} = -(a+1)-(a^2+3a+3)\rho+(a+2)\rho^2.
\]
It is then straightforward to compute that 
\[
\ID(A_1) = 2, \quad \ID(A_2) = 1,
\] 
and
\[
\IV(A_1) = 1, \quad \IV(A_2) = a^2+3a+3.
\]

By \Cref{thm:cii}, we see that $\iota_{\textrm{int}}(A_1) = 1$, and $\iota_{\textrm{int}}(A_2) = 0$. Meanwhile, since the edges of $A_2$ contains no other integral points other than the vertices, by Pick's formula we see that there are $\frac 12(a^2+3a+2)$ integral points in the interior of the face $A_2$, all of which are indecomposable. So we conclude that $\iota(K) = \frac 12(a^2+3a+6)$. This agrees with the characterisation of the indecomposable elements in $\mc O_K^+$ found in \cite{KT2023}. In particular, the unique indecomposable element that does not lie on the sail is given by $1+\rho+\rho^2$.
\end{expl}

Finally, we give a useful criterion to detect sails, which we will use in later sections.
\begin{lem}\label{lem:pis} 
Let $S$ be a polytope of dimension $n-1$ with vertices in $\mc O_K^+$. Suppose $S$ is contained in the hyperplane $\Tr_{K/\Q}(\delta(-)) = 1$ for some $\delta\in\mc O_K^{\vee,+}$. Then $S \subseteq \mc S_K$.
\end{lem}
\begin{proof}
Obviously the polytope $S$ is contained in the Klein polyhedron $\mc K_K$. On the other hand, we observe that $\Tr_{K/\Q}(\delta\alpha)\ge 1$ for every $\alpha\in\mc O_K^+$. So we see that $S \subseteq \partial \mc K_K = \mc S_K$.
\end{proof}

\section{Unit signature ranks and indecomposable elements}

\subsection{Units in biquadratic fields}\label{sect:bqu} 
Let us now turn to the case of biquadratic fields.
First we recall the following theorem of Kubota, which characterises for real biquadratic fields $K$ the group of units $\mc O_K^+$.

\begin{thm}[{\cite[Satz 1]{Kubota1956}}]\label{thm:k} 
	Let $K$ be a real biquadratic field, with quadratic subfields $K_1, K_2, K_3$. For $1\le i \le 3$, let $\varepsilon_i$ denote the fundamental unit of $K_i$. 
	\ben
	\item Suppose $\NN_{K_i/\Q}(\varepsilon_i) = 1$ for some $1\le i \le 3$. Then a system of fundamental units of $\mc O_K^\times$ is given by one of the following (up to relabelling of units $\varepsilon_i$):\\
	\NumTabs{2} \begin{enumerate*}[label=(\roman*),itemjoin=\tab]
		\item $\varepsilon_1,\varepsilon_2,\varepsilon_3$,
		\item $\sqrt{\varepsilon_1},\varepsilon_2,\varepsilon_3$,
		\item $\sqrt{\varepsilon_1},\sqrt{\varepsilon_2},\varepsilon_3$,
		\item $\sqrt{\varepsilon_1\varepsilon_2},\varepsilon_2,\varepsilon_3$,
		\item $\sqrt{\varepsilon_1\varepsilon_2},\sqrt{\varepsilon_3},\varepsilon_2$,
		\item $\sqrt{\varepsilon_1\varepsilon_2},\sqrt{\varepsilon_2\varepsilon_3},\sqrt{\varepsilon_3\varepsilon_1}$,
		\item $\sqrt{\varepsilon_1\varepsilon_2\varepsilon_3},\varepsilon_2,\varepsilon_3$,
	\end{enumerate*}\NumTabs{1}\\
	where the units $\varepsilon_i$ under the radical always have $\NN_{K_i/\Q}(\varepsilon_i) = 1$.
	\item Suppose $\NN_{K_i/\Q}(\varepsilon_i) = -1$ for every $1\le i \le 3$. Then a system of fundamental units of $\mc O_K^\times$ is given by one of the following:\\
	\NumTabs{2} \begin{enumerate*}[label=(\roman*),itemjoin=\tab]
		\item $\varepsilon_1,\varepsilon_2,\varepsilon_3$,
		\item $\sqrt{\varepsilon_1\varepsilon_2\varepsilon_3},\varepsilon_2,\varepsilon_3$.
	\end{enumerate*}\NumTabs{1}
	\ee
\end{thm}

\begin{lem}\label{lem:un1} 
	Let $K$ be a real biquadratic field with $\sgnrk(K) \le 3$. Then we have $\NN_{K/\Q}(\eta) = 1$ for every $\eta\in\mc O_K^\times$.
\end{lem}
\begin{proof}
	First suppose $\NN_{K_i/\Q}(\varepsilon_i) = 1$ for some $1\le i \le 3$. Note that this assumption actually implies $\sgnrk(K) \le 3$ (see \cite[Remark 1]{DDK2019}). We see from above that every $\eta\in\mc O_K^\times$ is of the form $\eta = \pm \sqrt{\varepsilon_1^{m_1}\varepsilon_2^{m_2}\varepsilon_3^{m_3}}$, where $m_i\in\Z$ is even unless $\NN_{K_i/\Q}(\varepsilon_i) = 1$.
	
	Let $\sigma_i$ denote the non-trivial element in $\Gal(K/K_i)$, and write $\eta^{\sigma_i} := \sigma_i(\eta)$. In particular, we have 
	\begin{equation*}
	\varepsilon_j^{\sigma_i} = \sigma_i(\varepsilon_j) = \begin{cases} \varepsilon_j & \text{ if } i=j,\\ \NN_{K_j/\Q}(\varepsilon_j) & \text{ if } i\ne j.\end{cases}
	\end{equation*}
	Then we have
	\begin{align*}
		\eta^{2(1+\sigma_1)} &= \varepsilon_1^{2m_1} \implies \eta^{1+\sigma_1} = (-1)^{\nu_1}\varepsilon_1^{m_1} \implies \eta^{\sigma_1} = (-1)^{\nu_1} \varepsilon_1^{m_1} \eta^{-1},\\
		\eta^{2(1+\sigma_2)} &= \varepsilon_2^{2m_2} \implies  \eta^{1+\sigma_2} = (-1)^{\nu_2}\varepsilon_2^{m_2} \implies \eta^{\sigma_2} = (-1)^{\nu_2} \varepsilon_2^{m_2} \eta^{-1},\\
	\end{align*}
	for some $\nu_1,\nu_2\in\Z$, and
	\begin{align*}
		\eta^{\sigma_3} = \eta^{\sigma_1\sigma_2} = \rb{(-1)^{\nu_1}\varepsilon_1^{m_1} \eta^{-1}}^{\sigma_2} = (-1)^{\nu_1+\nu_2} \varepsilon_1^{-m_1} \varepsilon_2^{-m_2} \eta.
	\end{align*}
	This implies
	\begin{equation*}
		\NN_{K/\Q}(\eta) = \eta^{1+\sigma_1+\sigma_2+\sigma_3} = \eta \rb{(-1)^{\nu_1}\varepsilon_1^{m_1}\eta^{-1}} \rb{(-1)^{\nu_2}\varepsilon_2^{m_2}\eta^{-1}} \rb{(-1)^{\nu_1+\nu_2}\varepsilon_1^{-m_1}\varepsilon_2^{-m_2} \eta} = 1.
	\end{equation*}
	
	Now suppose $\NN_{K_i/\Q}(\varepsilon_i) = -1$ for every $1\le i \le 3$. In this case, we observe that the signatures of $-1$, $\varepsilon_1$, and $\varepsilon_2$ are linear independent over $\F_2$. By the assumption $\sgnrk(K) \le 3$, this says the unit signature group of $K$ is generated by these signatures. Since we have
	\[
	\NN_{K/\Q}(-1) = \NN_{K/\Q}(\varepsilon_1) = \NN_{K/\Q}(\varepsilon_2) = 1,
	\]
	the lemma follows.
\end{proof}

\begin{lem}\label{lem:usc} 
	Let $K$ be a real biquadratic field with $\sgnrk{K}\ge 3$. Let $K_i = \Q(\sqrt{D_i})$, $1\le i\le 3$, be the quadratic subfields of $K$, with $D_i>1$ squarefree. Then for each $1\le i \le 3$, there exists a unit $\eta_i \in \mc O_K^\times$ with the same signature as $\sqrt{D_i}$.
\end{lem}
\begin{proof}
	If $\sgnrk(K) = 4$, then the statement is obvious because $\mc O_K^\times$ has all signatures. Now suppose $\sgnrk(K) = 3$. By \Cref{lem:un1}, we see that for $\eta\in\mc O_K^\times$, the signature $\sgn(\eta)$ lies in the kernel of the surjective map
	\[
	\NN: \F_2^4 \to \F_2, \quad (s_1,\ldots, s_4) \mapsto \sum_i s_i.
	\]
	Since $\sgnrk(K) = 3 = \dim_{\F_2} \ker(\NN)$, we see that $\mc O_K^\times$ has all signatures in $\ker(\NN)$. The statement then follows, because $\NN_{K/\Q}(\sqrt{D_i}) = (\NN_{K_i/\Q}(\sqrt{D_i}))^2  > 0$, which implies $\sgn(\sqrt{D_i}) \in \ker(\NN)$. 
\end{proof}

\subsection{Continued fraction}\label{sect:cf}
Now we recall some well-known results about indecomposable elements in real quadratic fields and continued fractions. See \cite{JW2009,Khinchin1964,Perron1977} for detailed discussion on the subject. Let $K = \Q(\sqrt{D})$ be a real quadratic field, with $D>1$ squarefree. We write
\[
\omega_D := \begin{cases} \sqrt{D} & \text{if } D\equiv 2,3\pmod{4},\\ \frac{1+\sqrt{D}}2 & \text{if } D\equiv 1\pmod{4}.\end{cases}
\]
Then the ring of integers $\mc O_K$ is given by $\Z[\omega_D]$. As a quadratic irrational number, $-\ol{\omega_D}$ admits a periodic continued fraction
\[
-\ol{\omega_D} = [u_0; \ol{u_1,\ldots, u_s}].
\]
In particular, we have
\begin{equation}\label{eq:us} 
	u_s = \begin{cases} 2u_0 & \text{if } D\equiv 2,3\pmod{4},\\ 2u_0+1 & \text{if } D\equiv 1\pmod{4}.\end{cases}
\end{equation}
From the continued fraction, we obtain a series of \emph{convergents}
\[
\frac{s_i}{t_i} := [u_0;u_1,\ldots,u_i], \quad i\ge 0.
\]
As a convention, we define $(s_{-1},t_{-1}) := (1,0)$. We associate with these convergents elements in $\mc O_K$, which by an abuse of notation are also called convergents:
\begin{equation}\label{eq:conv} 
	\beta_i := s_i + t_i \omega_D \in \mc O_K, \quad i\ge -1.
\end{equation}
We also define the \emph{semiconvergents}
\begin{equation}\label{eq:sc} 
	\beta_{i,l} := s_{i,l}+t_{i,l}\omega_D := (s_i+ls_{i+1}) + (t_i+lt_{i+1})\omega_D, \quad i\ge -1, \quad 0\le l\le u_{i+2}.
\end{equation}
The convergents $\beta_i$ (resp. semiconvergents $\beta_{i,l}$) are called \emph{upper} if $i$ is odd, and \emph{lower} if $i$ is even. We note that the upper (resp. lower) semiconvergents gives upper (resp. lower) bounds for $-\ol{\omega_D}$, and that the upper semiconvergents (and their conjugates) are precisely the indecomposable elements in $\mc O_K^+$ \cite{DS}.

Let $\tau_+,\tau_- : K\to\R$ be the two real embeddings which sends $\sqrt{D}$ to $\sqrt{D}$ and $-\sqrt{D}$ respectively. With respect to these embeddings, the upper semiconvergents have signature $(+,+)$, while the lower semiconvergents have signature $(+,-)$. For $i\ge -1$, we also define
\begin{equation}\label{eq:cde} 
	\delta_i := \begin{cases} \frac{(-1)^i}{2\sqrt{D}}(t_i\ol\omega_D-s_i) & \text{if } D\equiv 2,3\pmod{4},\\ \frac{(-1)^i}{\sqrt{D}}(t_i\ol\omega_D-s_i) & \text{if } D\equiv 1\pmod{4}.\end{cases}.
\end{equation}
It is easy to verify that $\delta_i$ lies in the codifferent $\mc O_K^\vee$, and $\delta_{i+1}$ has the same signature as $\beta_{i,l}$. Moreover, we have $\Tr_{K/\Q}(\delta_{i+1}\beta_{i,l}) = 1$ for every $0\le l \le u_{i+2}$. 

\subsection{Universal lattices and indecomposable elements}\label{sect:lattices}

Let $K$ be a totally real number field. We shall consider \emph{quadratic $\mc O_K$-lattices} $(\Lambda,Q)$ over $K$, that is, a finitely generated $\mc O_K$-module $\Lambda$ equipped with a quadratic form $Q$. The \emph{rank} of $(\Lambda,Q)$ is the $K$-dimension of $K\Lambda$. The lattice $(\Lambda,Q)$ is \emph{integral} if $Q(v) \in \mc O_K$ for every $v\in\Lambda$, totally positive if $Q$ is totally positive, and \emph{classical} if all the values in the associated symmetric bilinear form lie in $\mc O_K$. An integral, totally positive lattice $(\Lambda,Q)$ is \emph{universal} if it represents every element in $\mc O_K^+$. 

It is known that classical universal lattices exist over every totally real number field \cite{HKK1978}, so we may denote by $R_{\cls}(K)$ the minimal rank of a classical universal lattice. The following theorems say that $R_{\cls}(K)$ can be used to give a lower bound for $\iota(K)$. To state the theorem, we recall that the \emph{Pythagoras number} of a ring $R$ is the smallest integer $s(R)$ such that every sum of squares of elements of $R$ can be expressed as the sum of $s(R)$ squares.
\begin{thm}[{\cite[Proposition 7.1]{KT2023}}]\label{thm:mmr} 
	Let $K$ be a totally real number field of degree $n = [K:\Q]$. Let $s = s(\mc O_K)$ denote the Pythagoras number of $\mc O_K$, and $\mc S$ denote a set of representatives of classes of indecomposables in $\mc O_K$ up to multiplication by squares of units $\mc O_K^{\times 2}$. Then the quadratic form
	\[
	\sum_{\sigma\in\mc S} \sigma \rb{x_{1,\sigma}^2 + x_{2,\sigma}^2 + \ldots + x_{s,\sigma}^2}
	\] 
	is universal over $\mc O_K$, and has rank $s \cdot \#\mc S$. In particular, this implies $R_{\cls}(K) \le 2^{n-\sgnrk(K)} s \cdot \iota(K)$.
\end{thm}

\begin{thm}[{\cite[Corollary 3.3]{KY2021}}]\label{thm:pn} 
	Let $K$ be a totally real number field of degree $n = [K:\Q]$. Then the Pythagoras number $s(\mc O_K)$ is finite, and is bounded above by a function depending only on $n$. 
\end{thm}

Using \Cref{thm:mmr,thm:pn}, the following is immediate.
\begin{cor}\label{cor:rfif} 
	Let $K$ be a totally real number field. Then we have $R(K)\leq R_{\cls}(K) \ll \iota(K)$ for some constant depending only on $n = [K:\Q]$. 
\end{cor}

We shall also make use of the following result which bounds the number of short vectors in a $\Z$-lattice. The main part of the result is due to Regev--Stephens-Dawidovitz; for a reference to our formulation, see, e.g., the comments immediately following \cite[Theorem 3.1]{Man}.
\begin{prp}[{\cite[Theorem 1.1]{RDS}}]\label{prp:rs}
	Let $(\Lambda,Q)$ be a totally positive classical quadratic $\Z$-lattice of rank $R$, and $m\in\N$. Let $N_{\Lambda,Q}(m)$ denote the number of vectors $v\in\Lambda$ for which $Q(v) = m$. Then we have $N_{\Lambda,Q}(m) \le C(R,m)$, where
	\[
	C(R,m) \le \begin{cases} \max\cb{480,2R(R-1)} & \text{ if } m=2,\\ 2 \binom{R+2m-2}{2m-1} & \text{ otherwise.}\end{cases}
	\]
\end{prp}

\subsection{Proof of \texorpdfstring{\Cref{thm:usr}}{Theorem~1.1}}\label{subsec:pf 1.1}

Now we are ready to prove \Cref{thm:usr}.

\begin{proof}[Proof of Theorem $\ref{thm:usr}$]
	Let $K$ be a real biquadratic field with $\sgnrk(K) \ge 3$. Let $K_i = \Q(\sqrt{D_i})$, $1\le i \le 3$ be the quadratic subfields of $K$, with $D_i>1$ squarefree. Without loss of generality, we assume $D_1<D_2<D_3$. By \Cref{lem:usc}, there is a unit $\eta \in \mc O_K^\times$ with the same signature as $\sqrt{D_3}$. 
	
	We write 
	\[
	-\ol{\omega_{D_3}} = [u_0;\ol{u_1,\ldots,u_s}]
	\]
	the continued fraction of $-\ol{\omega_{D_3}}$, and define $u := \max\cbm{u_i}{i\in\N}$. We construct as in \eqref{eq:sc} the semiconvergents
	\[
	\beta_{i,l} = s_{i,l} + t_{i,l} \omega_{D_3} \in \mc O_{K_3}, \quad i\ge -1, \quad 0\le l \le u_{i+2},
	\]
	as well as the codifferent elements $\delta_i \in \mc O_{K_3}^\vee$, $i\ge -1$. For $i\ge -1$ odd, we observe that the elements $\beta_{i,l}$ and $\delta_{i+1}$ are totally positive, and we have 
	\[
	\Tr_{K/\Q}\rb{\delta_{i+1}\beta_{i,l}} = 2 \Tr_{K_3/\Q}\rb{\delta_{i+1}\beta_{i,l}} = 2 
	\]
	for $0\le l \le u_{i+2}$. Meanwhile, for $i\ge 0$ even, we observe that elements $\beta_{i,l}$ and $\delta_{i+1}$ have the same signature as $\sqrt{D_3}$. Hence the elements $\eta\beta_{i,l}$ and $\eta^{-1}\delta_{i+1}$ are totally positive, and we have
	\[
	\Tr_{K/\Q}\rb{\rb{\eta^{-1}\delta_{i+1}}\rb{\eta\beta_{i,l}}} = 2 \Tr_{K_3/\Q}\rb{\delta_{i+1}\beta_{i,l}} = 2
	\]
	for $0\le l \le u_{i+2}$. We thus conclude that there exists $\delta \in \mc O_K^{\vee,+}$ such that there are $u+1$ distinct elements $\gamma_1,\ldots,\gamma_{u+1} \in \mc O_K^+$ satisfying $\Tr_{K/\Q}(\delta\gamma_j) = 2$.
	
	Now let $(\Lambda,Q)$ be a classical universal $\mc O_K$-lattice of rank $R \ge R_{\cls}(K)$. By fixing a $\Z$-basis of $\mc O_K$, we can identify $\Lambda$ with a classical $\Z$-lattice of rank $4R$ via an isomorphism $\varphi: \Lambda \xrightarrow{\sim} \Z^{4R}$, on which we equip a quadratic form $q(v) := \Tr_{K/\Q}(\delta Q(\varphi^{-1}(v)))$, which is positive definite. Since $Q$ is universal, for each $1\le i \le u+1$ we can find $w_i \in \Lambda$ such that $Q(w_i) = \gamma_i$. For the corresponding vector $v_i := \varphi(w_i) \in \Z^{4R}$, we have
	\[
	q(\pm v_i) = \Tr_{K/\Q}\rb{\delta Q(w_i)} = \Tr_{K/\Q}(\delta\gamma_i) = 2.
	\]
	So there are at least $2(u+1)$ vectors $v\in\Z^{4R}$ for which we have $q(v) = 2$. By \Cref{prp:rs}, we deduce that
	\[
	u+1 \le \frac 12 C(4R,2) \le \max\cb{240, 4R(4R-1)} \ll R^2.
	\]
	On the other hand, from \eqref{eq:us} we see that
	\[
	u \ge u_s \ge 2 \flr{-\ol{\omega_{D_3}}} > \sqrt{D_3} - 3 \gg \sqrt{D_3}.
	\]
	It follows that $D_3 \ll R^4$. Since the discriminant $\Delta_K$ of $K$ satisfies (see, e.g., \cite{Chatelain1973,Schmal1989})
	\[
	\Delta_K \asymp D_1D_2D_3 \ll D_3^3,
	\]
	we conclude that $R_{\cls}(K) \gg \Delta_K^{1/12}$. By \Cref{cor:rfif} it follows that $\iota(K) \gg \Delta_K^{1/12}$ as well.
	
	Now suppose $(\Lambda,Q)$ be a non-classical universal $\mc O_K$-lattice of rank $R\ge R(K)$. Then $(\Lambda,2Q)$ is a classical $2\mc O_K$-universal lattice (i.e., it represents all elements in $\mc O_K^+ \cap 2\mc O_K$). Applying the aforementioned arguments with $2Q$ in place of $Q$, this yields $2(u+1)$ vectors in $v\in\Z^{4R}$ for which we have $q(v)=4$. By \Cref{prp:rs}, we deduce that
	\[
	u+1 \le \frac 12 C(4R,4) \le \binom{4R+6}{7} \ll R^7.
	\]
	Then the same argument gives $R(K) \gg \Delta_K^{1/42}$.
\end{proof}

\subsection{Finding unit signature ranks for biquadratic fields}

Here we give a brief account on finding the unit signature rank of a real biquadratic field $K$. Everything in this subsection is easily derived from the results of Kubota \cite{Kubota1956}. Nevertheless, we include this exposition for the reader's convenience.

For this we need some notations. Let $F = \Q(\sqrt{D})$ be a real quadratic field, and $\varepsilon \in \mc O_F^\times$ with $\NN_{F/\Q}(\varepsilon) = 1$. Then there exists a unique squarefree integer $d = d_F(\varepsilon)\in\N$ such that $d\varepsilon \in F^2$. In fact, this $d$ is given by the squarefree part of $\Tr_{F/\Q}(\varepsilon+1)$. One easily verifies that $\Tr_{F/\Q}(\varepsilon+1) \varepsilon = (\varepsilon+1)^2$, and it follows that $\sqrt{d\varepsilon} \in F$ is totally positive.

As in \Cref{sect:bqu}, let $K$ be a real biquadratic field, with quadratic subfields $K_1, K_2, K_3$ and corresponding fundamental units $\varepsilon_1,\varepsilon_2,\varepsilon_3$. The first thing to look at is the norm $\NN_{K_i/\Q}(\varepsilon_i)$ of the fundamental units. 

Suppose $N_{K_i/\Q}(\varepsilon_i) = -1$ for every $1\le i \le 3$. From \Cref{thm:k}, we see that if $\varepsilon_1\varepsilon_2\varepsilon_3$ is not a square in $K$, then a system of fundamental units of $\mc O_K^\times$ is given by $\varepsilon_1,\varepsilon_2,\varepsilon_3$, and we have $\sgnrk(K) = 3$. On the other hand, if $\varepsilon_1\varepsilon_2\varepsilon_3$ is a square in $K$, then \cite[Hilfssatz 3]{Kubota1956} shows that $\NN_{K/\Q}(\sqrt{\varepsilon_1\varepsilon_2\varepsilon_3}) = -1$. Thus $-1, \sqrt{\varepsilon_1\varepsilon_2\varepsilon_3}, \varepsilon_2, \varepsilon_3$ have independent signatures, and we have $\sgnrk(K) = 4$. To determine whether $\varepsilon_1\varepsilon_2\varepsilon_3$ is a square in $K$, we use the following criterion.

\begin{prp}[{\cite[Zusatz 1]{Kubota1956}}]
	Assume the settings as above, and suppose $N_{K_i/\Q}(\varepsilon_i) = -1$ for every $1\le i \le 3$. We define the following elements in $K$:
	\begin{align*}
		\xi_0 &= \varepsilon_1\varepsilon_2\varepsilon_3 + \varepsilon_1 + \varepsilon_2 - \varepsilon_3, & \xi_1 &= \varepsilon_1\varepsilon_2\varepsilon_3 + \varepsilon_1 - \varepsilon_2 + \varepsilon_3,\\
		\xi_2 &= \varepsilon_1\varepsilon_2\varepsilon_3 - \varepsilon_1 + \varepsilon_2 + \varepsilon_3, & \xi_3 &= \varepsilon_1\varepsilon_2\varepsilon_3 - \varepsilon_1 - \varepsilon_2 - \varepsilon_3.
	\end{align*}
	Then $\varepsilon_1\varepsilon_2\varepsilon_3 \in K^2$ if and only if $\Tr_{K/\Q}(\xi_i) \in K^2$ for some (or equivalently every) $i \in \cb{0,1,2,3}$. 
\end{prp}

Now suppose $\NN_{K_i/\Q}(\varepsilon_i) = 1$ for some $1\le i \le 3$. To determine the set of fundamental units of $\mc O_K$, we use the following criterion.
\begin{prp}[{\cite[Hilfssatz 11]{Kubota1956}}]
	Assume the settings as above, and suppose $N_{K_i/\Q}(\varepsilon_i) = 1$ for some $1\le i \le 3$. Let $\eta = \varepsilon_1^{m_1}\varepsilon_2^{m_2}\varepsilon_3^{m_3} \in \mc O_K^\times$, where $m_i\in\Z$ is even unless $\NN_{K_i/\Q}(\varepsilon_i) = 1$. Then $\eta \in K^2$ if and only if 
	\begin{equation}\label{eq:kc} 
		d_\eta := d_{K_1}(\varepsilon_1^{m_1}) d_{K_2}(\varepsilon_2^{m_2}) d_{K_3}(\varepsilon_3^{m_3}) \in K^2.
	\end{equation}
\end{prp}

To find the unit signature rank of $K$, it remains to determine the signature of $\sqrt{\eta}$ for which $\eta\in K^2$. Write $K_i = \Q(\sqrt{D_i})$, $1\le i \le 3$, with $D_i > 1$ squarefree. Since $\sqrt{d_{K_i}(\varepsilon_i^{m_i})\varepsilon_i^{m_i}}$ is totally positive, it follows that $\sqrt{\eta}$ has the same signature as $\sqrt{d_\eta}$. Finally, \eqref{eq:kc} implies $d_\eta \in c_\eta \Q^2$ for some $c_\eta \in \cb{1,D_1,D_2,D_3}$, this says $\sgn(\sqrt{\eta}) = \sgn(\sqrt{c_\eta})$. 

\subsection{Applications towards Kitaoka conjecture}

The Kitaoka conjecture states that there are only finitely many totally real number fields admit universal ternary forms. Towards this conjecture, Krásenský, Tinková, and Zemková \cite{KTZ2020} showed the following result concerning biquadratic fields.
\begin{thm}[\cite{KTZ2020}]\label{thm:bqk} 
	There are no classical universal ternary forms over real biquadratic fields.
\end{thm}

In their proof, the case where the unit signature rank of $K$ is $\le 2$ is quite easy; however, for the case where $\sgnrk(K)\ge 3$, they relied on a lengthy and technical proof to show that there are only finitely many possible exceptions that one have to check. Our \Cref{thm:usr} gives a vast simplification for this part of the proof. 

\begin{cor}
	Let $K$ be a real biquadratic field with $\sgnrk(K)\ge 3$, and let $K_i = \Q(\sqrt{D_i})$, $1\le i \le 3$ be the quadratic subfields of $K$, with $D_i > 1$ squarefree. If $K$ admits a ternary universal form, then $D_i < 133^2$.
\end{cor}
\begin{proof}
	We follow the proof of \Cref{thm:usr}, and optimise the constants along the proof. For $R = 3$, we find $C(12) \le 264$, where the equality is attained for example by the lattices $\Z^{12}$, as well as $E_8 \oplus D_4$. It follows that $u\le 131$, and $\lfloor -\ol{\omega_{D_i}} \rfloor \le 65$. Using the definition of $\omega_{D_i}$, we obtain $\sqrt{D_i} < 133$, or equivalently $D_i < 133^2$, as claimed.
\end{proof}

\section{A family of fields with few indecomposables}

In this section, we let $F = \Q(\sqrt{5})$. Then $\frac{1+\sqrt{5}}2$ is a fundamental unit of $\mc O_F^\times$. For $n\in\N$, we define integers $x_n, y_n$ by the equation
\[
\frac{x_n+y_n\sqrt{5}}2 = \rb{\frac{1+\sqrt{5}}2}^n. 
\]
It is easily verified that $x_n,y_n$ satisfy the recurrence relations
\begin{equation}\label{eq:xy1} 
x_{n+1} = \frac{x_n+5y_n}2, \quad y_{n+1} = \frac{x_n+y_n}2,
\end{equation}
as well as the equations
\begin{equation}\label{eq:xy2} 
x_ny_{n+1} - x_{n+1}y_n = 2(-1)^n, \quad x_nx_{n+1} - 5y_ny_{n+1} = 2(-1)^n,
\end{equation}
and
\begin{equation}\label{eq:xy3}
x_n^2 - 5y_n^2 = 4(-1)^n.
\end{equation}

For $n\ge 0$ be an integer such that
\begin{equation}\label{eq:xy4}
p = p_n := y_{12n+3}^2 - 1, \quad r = r_n := x_{12n+3}^2-1
\end{equation}
are squarefree integers, and let $K = K_n := \Q(\sqrt{5},\sqrt{p})$. The discriminant of the field $K$ is given by \cite[Satz 2.1]{Schmal1989}
\begin{equation}\label{eq:pd} 
\Delta_K = 16(5p_n)^2 = 16(x_{24n+6}-3)^2.
\end{equation}
The subfields of $K$ are given by
\[
K_1 := F = \Q(\sqrt{5}), \quad K_2 := \Q(\sqrt{p}) \quad K_3 := \Q(\sqrt{r}).
\]
For convenience, we shall write $X := x_{12n+3}$, and $Y := y_{12n+3}$. It is easily verified that
\[
\epsilon_1 = \frac{3+\sqrt{5}}2, \quad \epsilon_2 = Y+\sqrt{p}, \quad \epsilon_3 = X+\sqrt{r}
\]
are the generators of $\mc O_{K_1}^{\times,+}$, $\mc O_{K_2}^{\times,+}$, $\mc O_{K_3}^{\times,+}$ respectively. We note from \cite{Williams1970} that an integral basis of $\mc O_K$ is given by
\[
\cb{1,\frac{1+\sqrt{5}}2, \sqrt{p}, \frac{\sqrt{p}+\sqrt{r}}2},
\]
and an integral basis of the codifferent $\mc O_K^\vee$ is given by
\[
\cb{\frac 14-\frac{\sqrt{5}}{20}, \frac{\sqrt{5}}{10}, \frac{\sqrt{p}}{4p}-\frac{\sqrt{r}}{4r}, \frac{\sqrt{r}}{2r}}.
\]
We shall write $[a,b,c,d]$ to denote the element $a+b\sqrt{5}+c\sqrt{p}+d\sqrt{r} \in K$.

\begin{thm}\label{thm:pfi} 
Let $K = K_n = \Q(\sqrt{5}, \sqrt{p_n})$ be a biquadratic field constructed as above. A complete set of indecomposable elements in $\mc O_K^+$ modulo totally positive units is given by
\[
1, \quad \rho := \sqb{\frac{X+Y}2, 0, -\frac 12, \frac 12}, \quad \text{ and}
\]
\[
\gamma_j = \sqb{\frac{y_{2j-1}X+1}2, \frac{x_{2j-1}X-(-1)^j}{10}, (-1)^j\frac{x_{2j-1}}2, (-1)^j\frac{y_{2j-1}}2}, \quad 1\le j \le 6n+1.
\]
So $\iota(K) = 6n+3$. In particular, for this family of biquadratic fields, we have
\[
\frac{\log \Delta_K}{8w} \le \iota(K) \le \frac{\log \Delta_K}{8w} + 1,
\]
where $w := \log (\frac{1+\sqrt{5}}2)$. 
\end{thm}

The rest of the section is devoted to the proof of \Cref{thm:pfi}. We fix the embeddings
\begin{align*}
\tau_1\rb{a+b\sqrt{5}+c\sqrt{p}+d\sqrt{r}} &= a+b\sqrt{5}+c\sqrt{p}+d\sqrt{r},\\
\tau_2\rb{a+b\sqrt{5}+c\sqrt{p}+d\sqrt{r}} &= a-b\sqrt{5}+c\sqrt{p}-d\sqrt{r},\\
\tau_3\rb{a+b\sqrt{5}+c\sqrt{p}+d\sqrt{r}} &= a+b\sqrt{5}-c\sqrt{p}-d\sqrt{r},\\
\tau_4\rb{a+b\sqrt{5}+c\sqrt{p}+d\sqrt{r}} &= a-b\sqrt{5}-c\sqrt{p}+d\sqrt{r}.
\end{align*}

We consider the following elements in the codifferent $\mc O_K^\vee$:
\begin{align*}
\delta_A &:= \sqb{\frac 14, -\frac 1{20}, 0, -\frac{1}{5(X+Y)}},\\
\delta_{B,j}^\pm &:= \sqb{\frac 14, -\frac 1{20}, \pm\frac{X-x_{4j-1}}{20p}, \mp\frac{X-y_{4j-1}}{4r}}, & &1\le j \le 3n,\\
\delta_{C,j}^\pm &:= \sqb{\frac 14, \pm\frac 1{20}, -\frac{Y-y_{2j-1}}{4p}, \mp\frac{Y-x_{2j-1}}{4r}}, & &1\le j \le 6n+1.
\end{align*}

\begin{prp}
The elements $1,\rho,\gamma_j, \delta_A, \delta_{B,j}^\pm, \delta_{C,j}^\pm$ constructed above are totally positive.
\end{prp}
\begin{proof}
This is some routine checking, using rational approximations of $\sqrt{p}$ and $\sqrt{r}$ to give lower bounds to the embeddings. Here we show that the elements $\gamma_j$ are totally positive; the proofs for other elements are analogous. The embeddings $\tau_i$ sends $\gamma_j$ to
\begin{align*}
\tau_1(\gamma_j) &= \rb{\frac{y_{2j-1}X+1}2 + \frac{x_{2j-1}X-(-1)^j}{10} \sqrt{5}} + (-1)^j \rb{\frac{x_{2j-1}}2 + \frac{y_{2j-1}}2 \sqrt{5}}\sqrt{p},\\
\tau_2(\gamma_j) &= \rb{\frac{y_{2j-1}X+1}2 - \frac{x_{2j-1}X-(-1)^j}{10} \sqrt{5}} + (-1)^j \rb{\frac{x_{2j-1}}2 - \frac{y_{2j-1}}2 \sqrt{5}}\sqrt{p},\\
\tau_3(\gamma_j) &= \rb{\frac{y_{2j-1}X+1}2 + \frac{x_{2j-1}X-(-1)^j}{10} \sqrt{5}} - (-1)^j \rb{\frac{x_{2j-1}}2 + \frac{y_{2j-1}}2 \sqrt{5}}\sqrt{p},\\
\tau_4(\gamma_j) &= \rb{\frac{y_{2j-1}X+1}2 - \frac{x_{2j-1}X-(-1)^j}{10} \sqrt{5}} - (-1)^j \rb{\frac{x_{2j-1}}2 - \frac{y_{2j-1}}2 \sqrt{5}}\sqrt{p}.
\end{align*}
Since we have $\frac{x_{2j-1}}{y_{2j-1}} < \sqrt{5}$, we check that
\[\begin{aligned}
\tau_1(\gamma_j) &< \tau_3(\gamma_j), & \tau_4(\gamma_j) &< \tau_2(\gamma_j) & &\text{ for $1\le j \le 6n+1$ odd, and}\\
\tau_3(\gamma_j) &< \tau_1(\gamma_j), & \tau_2(\gamma_j) &< \tau_4(\gamma_j) & &\text{ for $2\le j \le 6n$ even.}\\
\end{aligned}\]
So it suffices to show that $\tau_1(\gamma_j), \tau_4(\gamma_j) > 0$ for $j$ odd, and $\tau_2(\gamma_j), \tau_3(\gamma_j) > 0$ for $j$ even. 

Using the bound $\sqrt{p} < Y$, we obtain for $j$ odd
\begin{equation}\label{eq:t1jo1} 
\tau_1(\gamma_j) > \rb{\frac{y_{2j-1}X+1}2 + \frac{x_{2j-1}X+1}{10} \sqrt{5}} - \rb{\frac{x_{2j-1}}2 + \frac{y_{2j-1}}2 \sqrt{5}}Y.
\end{equation}
So it suffices to show the right of \eqref{eq:t1jo1} is positive. By rearranging the terms, we see that this is equivalent to
\begin{equation}\label{eq:t1jo2} 
\rb{Y\sqrt{5}-X}\rb{x_{2j-1}+y_{2j-1}\sqrt{5}} < 1+\sqrt{5}.
\end{equation}
To see that \eqref{eq:t1jo2} is true, we observe that
\[
\rb{Y\sqrt{5}-X}\rb{x_{2j-1}+y_{2j-1}\sqrt{5}} = 4 \rb{\frac{1+\sqrt{5}}2}^{2j-12n-4} \le 4\rb{\frac{1+\sqrt{5}}2}^{-2} < 1+\sqrt{5},
\]
using the assumption $1\le j \le 6n+1$. Similarly, for $j$ odd we have
\begin{equation}\label{eq:t4jo1} 
\tau_4(\gamma_j) > \rb{\frac{y_{2j-1}X+1}2 - \frac{x_{2j-1}X+1}{10} \sqrt{5}} + \rb{\frac{x_{2j-1}}2 - \frac{y_{2j-1}}2 \sqrt{5}}Y.
\end{equation}
So it suffices to show the right of \eqref{eq:t4jo1} is positive. By rearranging the terms, this is equivalent to
\begin{equation}\label{eq:t4jo2} 
\rb{X-Y\sqrt{5}}\rb{x_{2j-1}-y_{2j-1}\sqrt{5}} < \sqrt{5}-1.
\end{equation}
To see that \eqref{eq:t4jo2} is true, we observe that
\[
\rb{X-Y\sqrt{5}}\rb{x_{2j-1}-y_{2j-1}\sqrt{5}} = 4\rb{\frac{1+\sqrt{5}}2}^{-2j-12n-2} \le 4\rb{\frac{1+\sqrt{5}}2}^{-12n-4} < \sqrt{5}-1,
\]
again using the assumption $1\le j \le 6n+1$.

On the other hand, we obtain for $j$ even
\begin{equation}\label{eq:t2je1} 
\tau_2(\gamma_j) > \rb{\frac{y_{2j-1}X+1}2 - \frac{x_{2j-1}X-1}{10}\sqrt{5}} + \rb{\frac{x_{2j-1}}2-\frac{y_{2j-1}}2 \sqrt{5}}Y.
\end{equation}
By rearranging the terms, the right of \eqref{eq:t2je1} being positive is equivalent to
\[
\rb{X-Y\sqrt{5}}\rb{x_{2j-1}-y_{2j-1}\sqrt{5}} < 1+\sqrt{5},
\]
which follows from \eqref{eq:t4jo2}. Similarly, for $j$ even we have
\begin{equation}\label{eq:t3je1} 
\tau_3(\gamma_j) > \rb{\frac{y_{2j-1}X+1}2 + \frac{x_{2j-1}X-1}{10}\sqrt{5}} - \rb{\frac{x_{2j-1}}2+\frac{y_{2j-1}}2 \sqrt{5}}Y.
\end{equation}
By rearranging the terms, the right of \eqref{eq:t3je1} being positive is equivalent to
\begin{equation}\label{eq:t3je2} 
\rb{Y\sqrt{5}-X}\rb{x_{2j-1}+y_{2j-1}\sqrt{5}} < \sqrt{5}-1.
\end{equation}
To see that \eqref{eq:t3je2} is true, we observe that
\[
\rb{Y\sqrt{5}-X}\rb{x_{2j-1}+y_{2j-1}\sqrt{5}} = 4 \rb{\frac{1+\sqrt{5}}2}^{2j-12n-4} \le 4\rb{\frac{1+\sqrt{5}}2}^{-4} \le \sqrt{5}-1
\]
using the assumption $2\le j \le 6n$. Therefore $\gamma_j$ is totally positive for $1\le j \le 6n+1$.
\end{proof}

Next we show that the twisted trace-$1$ hyperplanes associated to $\delta_A, \delta_{B,j}^\pm, \delta_{C,j}^\pm$ contain certain translates of $1,\rho,\gamma_j$. As a convention we shall take $\gamma_0 := \gamma_1 \epsilon_1^{-1}\epsilon_3$.

\begin{prp}\label{prp:t1} 
The hyperplane $\Tr_{K/\Q}(\delta_A(-)) = 1$ contains the elements
\[
1, \epsilon_1, \epsilon_1^{6n+1}\epsilon_2, \epsilon_1^{6n+2}\epsilon_2, \epsilon_1^{6n+1}\epsilon_3, \epsilon_1^{6n+2}\epsilon_3, \epsilon_2\epsilon_3, \epsilon_1\epsilon_2\epsilon_3, \rho\epsilon_2, \rho\epsilon_1\epsilon_2, \gamma_{6n}\epsilon_1, \gamma_{6n}\epsilon_1^2, \gamma_{6n+1}\epsilon_2\epsilon_3, \gamma_{6n+1}\epsilon_1\epsilon_2\epsilon_3;
\]
for $1\le j \le 3n$, the hyperplanes $\Tr_{K/\Q}(\delta_{B,j}^+(-)) = 1$ contains the elements
\[
1, \epsilon_1, \epsilon_1^{2j-1}\epsilon_3, \epsilon_1^{2j}\epsilon_3, \gamma_{2j-2}\epsilon_1, \gamma_{2j-2}\epsilon_1^2, \gamma_{2j}, \gamma_{2j}\epsilon_1;
\]
for $1\le j \le 3n$, the hyperplanes $\Tr_{K/\Q}(\delta_{B,j}^-(-)) = 1$ contains the elements
\[
1, \epsilon_1, \epsilon_1^{2j-1}\epsilon_3^{-1}, \epsilon_1^{2j}\epsilon_3^{-1}, \gamma_{2j-1}, \gamma_{2j-1}\epsilon_1, \gamma_{2j+1}\epsilon_1^{-1}, \gamma_{2j+1};
\]
for $1\le j \le 6n+1$, the hyperplanes $\Tr_{K/\Q}(\delta_{C,j}^+(-)) = 1$ contains the elements
\[
\epsilon_1^{-1}, 1, \epsilon_1^{-j}\epsilon_2, \epsilon_1^{-j+1}\epsilon_2;
\]
for $1\le j \le 6n+1$, the hyperplanes $\Tr_{K/\Q}(\delta_{C,j}^-(-)) = 1$ contains the elements
\[
1, \epsilon_1, \epsilon_1^{j-1}\epsilon_2, \epsilon_1^j\epsilon_2.
\]
\end{prp}
\begin{proof}
This is also routine checking, using the relations \eqref{eq:xy1}-\eqref{eq:xy3} between $x_i$ and $y_i$. Here we show the statement for $\delta_A$; the proofs for other statements are analogous. For $\alpha = [a,b,c,d]$, we compute 
\begin{equation}\label{eq:tda} 
\Tr_{K/\Q}(\delta_A \alpha) = a-b-\frac{4dr}{5(X+Y)}. 
\end{equation}
Using \eqref{eq:tda}, we compute the trace as follows.
\begin{itemize}[wide]
\item For $\alpha = 1 = [1,0,0,0]$, we get $\Tr_{K/\Q}(\delta_A) = 1$.

\item For $\alpha = \epsilon_1 = [\frac 32,\frac 12,0,0]$, we get $\Tr_{K/\Q}(\delta_A\epsilon_1) = \frac 32 - \frac 12 = 1$.

\item For $\alpha = \epsilon_1^{6n+1}\epsilon_2 = [\frac{x_{12n+2}y_{12n+3}}2, \frac{y_{12n+2}y_{12n+3}}2, \frac{x_{12n+2}}2, \frac{y_{12n+2}}2]$, we get
\begin{multline}\label{eq:tda_s2} 
\Tr_{K/\Q}(\delta_A \epsilon_1^{6n+1}\epsilon_2) = \frac{x_{12n+2}y_{12n+3}}2 -\frac{y_{12n+2}y_{12n+3}}2 -\frac{2y_{12n+2}(x_{12n+3}^2-1)}{5(x_{12n+3}+y_{12n+3})}\\
\scalebox{0.87}{$\displaystyle =\frac{y_{12n+3}(x_{12n+2}x_{12n+3}-5y_{12n+2}y_{12n+3})+(4x_{12n+3}+5y_{12n+3})(x_{12n+2}y_{12n+3}-x_{12n+3}y_{12n+2})+4y_{12n+2}}{10(x_{12n+3}+y_{12n+3})}.$}
\end{multline}
Using \eqref{eq:xy1} and \eqref{eq:xy2}, the numerator in \eqref{eq:tda_s2} equals
\[
8x_{12n+3} + 4y_{12n+2} + 12y_{12n+3} = 10x_{12n+3} + 10y_{12n+3}.
\]
So we get $\Tr_{K/\Q}(\delta_A \epsilon_1^{6n+1}\epsilon_2) = 1$.

\item For $\alpha = \epsilon_1^{6n+2}\epsilon_2 = [\frac{x_{12n+3}y_{12n+4}}2, \frac{y_{12n+3}y_{12n+4}}2, \frac{x_{12n+4}}2, \frac{y_{12n+4}}2]$, we get
\begin{multline}\label{eq:tda_b2} 
\Tr_{K/\Q}(\delta_A \epsilon_1^{6n+2}\epsilon_2) = \frac{x_{12n+3}y_{12n+4}}2 -\frac{y_{12n+3}y_{12n+4}}2 -\frac{2y_{12n+4}(x_{12n+3}^2-1)}{5(x_{12n+3}+y_{12n+3})}\\
\scalebox{0.86}{$\displaystyle =\frac{y_{12n+3}(x_{12n+3}x_{12n+4}-5y_{12n+3}y_{12n+4})+(4x_{12n+3}+5y_{12n+3})(x_{12n+4}y_{12n+3}-x_{12n+3}y_{12n+4})+4y_{12n+4}}{10(x_{12n+3}+y_{12n+3})}.$}
\end{multline}
Using \eqref{eq:xy1} and \eqref{eq:xy2}, the numerator in \eqref{eq:tda_b2} equals
\[
8x_{12n+3} + 8y_{12n+3} + 4y_{12n+4} = 10x_{12n+3} + 10y_{12n+3}.
\]
So we get $\Tr_{K/\Q}(\delta_A \epsilon_1^{6n+2}\epsilon_2) = 1$.

\item For $\alpha = \epsilon_1^{6n+1}\epsilon_3 = [\frac{x_{12n+2}x_{12n+3}}2,\frac{x_{12n+3}y_{12n+2}}2,\frac{5y_{12n+2}}2,\frac{x_{12n+2}}2]$, we get
\begin{multline}\label{eq:tda_s3} 
\Tr_{K/\Q}(\delta_A\epsilon_1^{6n+1}\epsilon_3) = \frac{x_{12n+2}x_{12n+3}}2 - \frac{x_{12n+3}y_{12n+2}}2 - \frac{2x_{12n+2}(x_{12n+3}^2-1)}{5(x_{12n+3}+y_{12n+3})}\\
= \frac{x_{12n+3}(x_{12n+2}x_{12n+3}-5y_{12n+2}y_{12n+3}+5x_{12n+2}y_{12n+3}-5x_{12n+3}y_{12n+2}) + 4x_{12n+2}}{10(x_{12n+3}+y_{12n+3})}.
\end{multline}
Using \eqref{eq:xy1} and \eqref{eq:xy2}, the numerator in \eqref{eq:tda_s3} equals
\[
12x_{12n+3} + 4x_{12n+2} = 10x_{12n+3} + (2x_{12n+3}+4x_{12n+2}) = 10x_{12n+3}+ 10y_{12n+3}.
\]
So we get $\Tr_{K/\Q}(\delta_A \epsilon_1^{6n+1} \epsilon_3) = 1$.

\item For $\alpha = \epsilon_1^{6n+2}\epsilon_3 = [\frac{x_{12n+3}x_{12n+4}}2, \frac{x_{12n+3}y_{12n+4}}2, \frac{5y_{12n+4}}2, \frac{x_{12n+4}}2]$, we get
\begin{multline}\label{eq:tda_b3} 
\Tr_{K/\Q}(\delta_A\epsilon_1^{6n+2}\epsilon_3) = \frac{x_{12n+3}x_{12n+4}}2 -\frac{x_{12n+3}y_{12n+4}}2 - \frac{2x_{12n+4}(x_{12n+3}^2-1)}{5(x_{12n+3}+y_{12n+3})}\\
=\frac{x_{12n+3}(x_{12n+3}x_{12n+4}-5y_{12n+3}y_{12n+4}+5x_{12n+4}y_{12n+3}-5x_{12n+3}y_{12n+4}) + 4x_{12n+4}}{10(x_{12n+3}+y_{12n+3})}.
\end{multline}
Using \eqref{eq:xy1} and \eqref{eq:xy2}, we simplify the numerator in \eqref{eq:tda_b3} to
\[
8x_{12n+3} + 4x_{12n+4} = 10x_{12n+3} + (4x_{12n+4}-2x_{12n+3}) = 10x_{12n+3} + 10y_{12n+3}.
\]
So we get $\Tr_{K/\Q}(\delta_A \epsilon_1^{6n+2}\epsilon_3) = 1$.

\item For $\alpha = \epsilon_2\epsilon_3 = [x_{12n+3}y_{12n+3}, y_{12n+3}^2-1, x_{12n+3}, y_{12n+3}]$, we get
\begin{multline}\label{eq:tda_s23} 
\Tr_{K/\Q}(\delta_A\epsilon_2\epsilon_3) = x_{12n+3}y_{12n+3} - y_{12n+3}^2+1 - \frac{4y_{12n+3}(x_{12n+3}^2-1)}{5(x_{12n+3}+y_{12n+3})}\\
= \frac{y_{12n+3}(x_{12n+3}^2 - 5y_{12n+3}^2) + 5x_{12n+3} + 9y_{12n+3}}{5(x_{12n+3}+y_{12n+3})}.
\end{multline}
By \eqref{eq:xy3}, the numerator in \eqref{eq:tda_s23} equals $5x_{12n+3}+5y_{12n+3}$, thus $\Tr_{K/\Q}(\delta_A\epsilon_2\epsilon_3) = 1$.

\item For $\alpha = \epsilon_1\epsilon_2\epsilon_3 = [x_{12n+4}y_{12n+4}-\frac 32, y_{12n+4}^2-\frac 12, x_{12n+5}, y_{12n+5}]$, we get
\[
\Tr_{K/\Q}(\delta_A\epsilon_1\epsilon_2\epsilon_3) = x_{12n+4}y_{12n+4} -y_{12n+4}^2 -1 -\frac{4y_{12n+5}(x_{12n+3}^2-1)}{5(x_{12n+3}+y_{12n+3})}.
\]
Using \eqref{eq:xy1} and rewrite everything in terms of $x_{12n+3}$ and $y_{12n+3}$, we get
\begin{equation}\label{eq:tda_b23} 
\Tr_{K/\Q}(\delta_A\epsilon_1\epsilon_2\epsilon_3) = \frac{(2x_{12n+3}+y_{12n+3})(5y_{12n+3}^2-x_{12n+3}^2)-3x_{12n+3}+y_{12n+3}}{5(x_{12n+3}+y_{12n+3})}.
\end{equation}
By \eqref{eq:xy3}, the numerator in \eqref{eq:tda_b23} equals $5x_{12n+3}+5y_{12n+3}$, thus $\Tr_{K/\Q}(\delta_A\epsilon_1\epsilon_2\epsilon_3)=1$.

\item For $\alpha = \rho\epsilon_2 = [\frac{x_{12n+3}y_{12n+3}+1}2, \frac{y_{12n+3}^2-1}2, \frac{x_{12n+3}}2, \frac{y_{12n+3}}2]$, we get
\begin{multline}\label{eq:tda_sr2} 
\Tr_{K/\Q}(\delta_A\rho\epsilon_2) = \frac{x_{12n+3}y_{12n+3}+1}2 - \frac{y_{12n+3}^2-1}2 - \frac{2y_{12n+3}(x_{12n+3}^2-1)}{5(x_{12n+3}+y_{12n+3})}\\
=\frac{y_{12n+3}(x_{12n+3}^2-5y_{12n+3}^2)+10x_{12n+3}+14y_{12n+3}}{10(x_{12n+3}+y_{12n+3})}.
\end{multline}
By \eqref{eq:xy3}, the numerator in \eqref{eq:tda_sr2} equals $10x_{12n+3}+10y_{12n+3}$, thus $\Tr_{K/\Q}(\delta_A \rho\epsilon_2) = 1$.

\item For $\alpha = \rho\epsilon_1\epsilon_2 = [\frac{x_{12n+4}y_{12n+4}}2, \frac{y_{12n+4}^2}2, \frac{x_{12n+5}}2, \frac{y_{12n+5}}2]$, we get
\[
\Tr_{K/\Q}(\delta_A\rho\epsilon_1\epsilon_2) = \frac{x_{12n+4}y_{12n+4}}2 - \frac{y_{12n+4}^2}2 - \frac{2y_{12n+5}(x_{12n+3}^2-1)}{5(x_{12n+3}+y_{12n+3})}.
\]
Using \eqref{eq:xy1} and rewrite everything in terms of $x_{12n+3}$ and $y_{12n+3}$, we get
\begin{equation}\label{eq:tda_br2} 
\Tr_{K/\Q}(\delta_A\rho\epsilon_1\epsilon_2) = \frac{(2x_{12n+3}+y_{12n+3})(5y_{12n+3}^2-x_{12n+3}^2)+2x_{12n+3}+6y_{12n+3}}{10(x_{12n+3}+y_{12n+3})}.
\end{equation}
By \eqref{eq:xy3}, the numerator in \eqref{eq:tda_br2} equals $10x_{12n+3}+10y_{12n+3}$, thus $\Tr_{K/\Q}(\delta_A\rho\epsilon_1\epsilon_2) = 1$.

\item For $\alpha = \gamma_{6n}\epsilon_1 = [\frac{x_{12n+2}y_{12n+2}}2+1,\frac{y_{12n+2}^2}2, \frac{x_{12n+1}}2, \frac{y_{12n+1}}2]$, we get
\[
\Tr_{K/\Q}(\delta_A\gamma_{6n}\epsilon_1) = \frac{x_{12n+2}y_{12n+2}}2+1 - \frac{y_{12n+2}^2}2 - \frac{2y_{12n+1}(x_{12n+3}^2-1)}{5(x_{12n+3}+y_{12n+3})}. 
\]
Using \eqref{eq:xy1} and rewrite everything in terms of $x_{12n+1}$ and $y_{12n+1}$, we get
\begin{equation}\label{eq:tda_sg} 
\Tr_{K/\Q}(\delta_A\gamma_{6n}\epsilon_1) = 1+\frac{y_{12n+1}(x_{12n+1}^2-5y_{12n+1}^2)+4y_{12n+1}}{20(x_{12n+1}+2y_{12n+1})}. 
\end{equation}
By \eqref{eq:xy3}, the rightmost term on the right of \eqref{eq:tda_sg} vanishes, thus $\Tr_{K/\Q}(\delta_A\gamma_{6n}\epsilon_1) = 1$.

\item For $\alpha = \gamma_{6n}\epsilon_1^2 = [\frac{x_{12n+3}y_{12n+3}}2+1, \frac{y_{12n+3}^2}2, \frac{x_{12n+3}}2, \frac{y_{12n+3}}2]$, we get
\begin{multline}\label{eq:tda_bg} 
\Tr_{K/\Q}(\delta_A\gamma_{6n}\epsilon_1^2) = \frac{x_{12n+3}y_{12n+3}}2+1 - \frac{y_{12n+3}^2}2 - \frac{2y_{12n+3}(x_{12n+3}^2-1)}{5(x_{12n+3}+y_{12n+3})}\\
= 1 + \frac{y_{12n+3}(x_{12n+3}^2-5y_{12n+3}^2)+4y_{12n+3}}{5(x_{12n+3}+y_{12n+3})}.
\end{multline}
By \eqref{eq:xy3}, the rightmost term on the right of \eqref{eq:tda_bg} vanishes, thus $\Tr_{K/\Q}(\delta_A\gamma_{6n}\epsilon_1^2) = 1$.

\item For $\alpha = \gamma_{6n+1}\epsilon_2\epsilon_3 = [\frac{x_{12n+3}y_{12n+4}+1}2, \frac{y_{12n+3}y_{12n+4}-1}2, \frac{x_{12n+4}}2, \frac{y_{12n+4}}2]$, we get
\begin{multline}\label{eq:tda_sg23} 
\Tr_{K/\Q}(\delta_A\gamma_{6n+1}\epsilon_2\epsilon_3) = \frac{x_{12n+3}y_{12n+4}+1}2 - \frac{y_{12n+3}y_{12n+4}-1}2 - \frac{2y_{12n+4}(x_{12n+3}^2-1)}{5(x_{12n+3}+y_{12n+3})}\\
= 1+\frac{y_{12n+4}(x_{12n+3}^2-5y_{12n+3}^2)+ 4y_{12n+4}}{10(x_{12n+3}+y_{12n+3})}.
\end{multline}
By \eqref{eq:xy3}, the rightmost term on the right of \eqref{eq:tda_sg23} vanishes. Thus $\Tr_{K/\Q}(\delta_A\gamma_{6n+1}\epsilon_2\epsilon_3) = 1$.

\item For $\alpha = \gamma_{6n+1}\epsilon_1\epsilon_2\epsilon_3 = [\frac{x_{12n+5}y_{12n+4}-1}2, \frac{y_{12n+4}y_{12n+5}-1}2, \frac{x_{12n+6}}2, \frac{y_{12n+6}}2]$, we get
\[
\Tr_{K/\Q}(\delta_A\gamma_{6n+1}\epsilon_1\epsilon_2\epsilon_3) = \frac{x_{12n+5}y_{12n+4}-1}2 -\frac{y_{12n+4}y_{12n+5}-1}2 - \frac{2y_{12n+6}(x_{12n+3}^2-1)}{5(x_{12n+3}+y_{12n+3})}.
\]
Using \eqref{eq:xy1} and rewrite everything in terms of $x_{12n+3}$ and $y_{12n+3}$, we get
\begin{equation}\label{eq:tda_bg23} 
\Tr_{K/\Q}(\delta_A\gamma_{6n+1}\epsilon_1\epsilon_2\epsilon_3) = \frac{(3x_{12n+3}+y_{12n+3})(5y_{12n+3}^2-x_{12n+3}^2) + 8x_{12n+3} + 16y_{12n+3}}{20(x_{12n+3}+y_{12n+3})}.
\end{equation}
By \eqref{eq:xy3}, the numerator in \eqref{eq:tda_bg23} equals $20x_{12n+3}+20y_{12n+3}$, thus $\Tr_{K/\Q}(\delta_A\gamma_{6n+1}\epsilon_1\epsilon_2\epsilon_3) = 1$. \qedhere
\end{itemize}
\end{proof}

\begin{prp}\label{prp:pfi} 
The elements $1,\rho,\gamma_j$ constructed above are indecomposable.
\end{prp}
\begin{proof}
This follows from the observation that if $\alpha \in \mc O_K^+$, and there exists $\delta \in \mc O_K^{\vee,+}$ such that $\Tr_{K/\Q}(\delta\alpha) = 1$, then $\alpha$ is indecomposable.
\end{proof}

Now it remains to show that $1,\rho,\gamma_j$ are the \emph{only} indecomposables. For this we utilise the theory of sails. Following \Cref{prp:t1}, we define the following $3$-polytopes:
\begin{align*}
A &:= \left\langle1, \epsilon_1, \epsilon_1^{6n+1}\epsilon_2, \epsilon_1^{6n+2}\epsilon_2, \epsilon_1^{6n+1}\epsilon_3, \epsilon_1^{6n+2}\epsilon_3, \epsilon_2\epsilon_3, \epsilon_1\epsilon_2\epsilon_3,\right.\\
&\hspace{2cm}\left.\rho\epsilon_2, \rho\epsilon_1\epsilon_2, \gamma_{6n}\epsilon_1, \gamma_{6n}\epsilon_1^2, \gamma_{6n+1}\epsilon_2\epsilon_3, \gamma_{6n+1}\epsilon_1\epsilon_2\epsilon_3\right\rangle,\\
B_j^+ &:= \pb{1, \epsilon_1, \epsilon_1^{2j-1}\epsilon_3, \epsilon_1^{2j}\epsilon_3, \gamma_{2j-2}\epsilon_1, \gamma_{2j-2}\epsilon_1^2, \gamma_{2j}, \gamma_{2j}\epsilon_1}, & &1\le j \le 3n,\\
B_j^- &:= \pb{1, \epsilon_1, \epsilon_1^{2j-1}\epsilon_3^{-1}, \epsilon_1^{2j}\epsilon_3^{-1}, \gamma_{2j-1}, \gamma_{2j-1}\epsilon_1, \gamma_{2j+1}\epsilon_1^{-1}, \gamma_{2j+1}}, & &1\le j \le 3n,\\
C_j^+ &:= \pb{\epsilon_1^{-1}, 1, \epsilon_1^{-j}\epsilon_2, \epsilon_1^{-j+1}\epsilon_2}, & &1\le j \le 6n+1,\\
C_j^- &:= \pb{1, \epsilon_1, \epsilon_1^{j-1}\epsilon_2, \epsilon_1^j\epsilon_2}, & &1\le j \le 6n+1.
\end{align*}

\begin{lem}\label{lem:pc} 
For the polytopes $A, B_j^\pm, C_j^\pm$, we have
\[
\IV(A) = 24, \quad \IV(B_j^\pm) = 9, \quad \IV(C_j^\pm) = 1.
\]
Moreover, these polytopes all admit a unimodular triangulation.
\end{lem}
\begin{proof}
It is more convenient to use alternative bases to describe these polytopes. For the polytope $A$, we take $1$ as the origin, and pick basis vectors $e_1 = \epsilon_1-1$, $e_2 = \epsilon_1^{6n+1}\epsilon_2-1$, and $e_3 = \epsilon_1^{6n+2}\epsilon_2-1$. We check that the integer simplex $\pb{1, \epsilon_1, \epsilon_1^{6n+1}\epsilon_2, \epsilon_1^{6n+2}\epsilon_2}$ corresponding to this basis has integral volume $1$. With respect to this basis, the lattice points on $A$ are taken to the following coordinates:
\begin{align*}
1&\mapsto(0,0,0)=:\texttt{A1}, & \epsilon_1&\mapsto(1,0,0)=:\texttt{A2},\\
\epsilon_1^{6n+1}\epsilon_2&\mapsto(0,1,0)=:\texttt{A3}, & \epsilon_1^{6n+2}\epsilon_2&\mapsto(0,0,1)=:\texttt{A4},\\
\epsilon_1^{6n+1}\epsilon_3&\mapsto(-2,-3,2) =: \texttt{A5}, & \epsilon_1^{6n+2}\epsilon_3&\mapsto(-2,-2,3)=:\texttt{A6},\\
\epsilon_2\epsilon_3&\mapsto(-2,-2,2)=:\texttt{A7}, & \epsilon_1\epsilon_2\epsilon_3 &\mapsto(-3,-2,4)=:\texttt{A8},\\
\rho\epsilon_2&\mapsto(-1,-1,1)=:\texttt{A9}, & \rho\epsilon_1\epsilon_2&\mapsto(-1,-1,2)=:\texttt{A10},\\
\gamma_{6n}\epsilon_1&\mapsto(-1,-2,1)=:\texttt{A11}, & \gamma_{6n}\epsilon_1^2&\mapsto(0,-1,1)=:\texttt{A12},\\
\gamma_{6n+1}\epsilon_2\epsilon_3&\mapsto(-1,0,1)=:\texttt{A13}, & \gamma_{6n+1}\epsilon_1\epsilon_2\epsilon_3&\mapsto(-2,-1,3)=:\texttt{A14}.
\end{align*}
With the new coordinates it is straightforward to verify that $\IV(A) = 24$. Now we give one (out of many possible) unimodular triangulation of $A$ into $24$ integer simplices with unit integer volume:
\begin{align*}
&\pb{\texttt{A1}, \texttt{A2}, \texttt{A3}, \texttt{A12}}, & &\pb{\texttt{A1}, \texttt{A2}, \texttt{A11}, \texttt{A12}}, & &\pb{\texttt{A1}, \texttt{A3}, \texttt{A8}, \texttt{A9}}, & &\pb{\texttt{A1}, \texttt{A3}, \texttt{A8}, \texttt{A14}},\\
&\pb{\texttt{A1}, \texttt{A3}, \texttt{A10}, \texttt{A12}}, & &\pb{\texttt{A1}, \texttt{A3}, \texttt{A10}, \texttt{A14}}, & &\pb{\texttt{A1}, \texttt{A9}, \texttt{A10}, \texttt{A12}}, & &\pb{\texttt{A1}, \texttt{A9}, \texttt{A10}, \texttt{A14}},\\
&\pb{\texttt{A1}, \texttt{A9}, \texttt{A11}, \texttt{A12}}, & &\pb{\texttt{A2}, \texttt{A3}, \texttt{A4}, \texttt{A10}}, & &\pb{\texttt{A2}, \texttt{A3}, \texttt{A10}, \texttt{A12}}, & &\pb{\texttt{A3}, \texttt{A4}, \texttt{A10}, \texttt{A14}},\\
&\pb{\texttt{A3}, \texttt{A7}, \texttt{A8}, \texttt{A9}}, & &\pb{\texttt{A3}, \texttt{A7}, \texttt{A8}, \texttt{A13}}, & &\pb{\texttt{A5}, \texttt{A6}, \texttt{A8}, \texttt{A11}}, & &\pb{\texttt{A5}, \texttt{A6}, \texttt{A11}, \texttt{A12}},\\
&\pb{\texttt{A5}, \texttt{A7}, \texttt{A8}, \texttt{A11}}, & &\pb{\texttt{A6}, \texttt{A8}, \texttt{A11}, \texttt{A12}}, & &\pb{\texttt{A7}, \texttt{A8}, \texttt{A9}, \texttt{A10}}, & &\pb{\texttt{A7}, \texttt{A8}, \texttt{A10}, \texttt{A11}},\\
&\pb{\texttt{A7}, \texttt{A9}, \texttt{A10}, \texttt{A11}}, & &\pb{\texttt{A8}, \texttt{A9}, \texttt{A10}, \texttt{A14}}, & &\pb{\texttt{A8}, \texttt{A10}, \texttt{A11}, \texttt{A12}}, & &\pb{\texttt{A9}, \texttt{A10}, \texttt{A11}, \texttt{A12}}.
\end{align*}

For the polytopes $B_j^+$, we take $1$ as the origin, and pick basis vectors $e_1 = \epsilon_1-1$, $e_2 = \epsilon_1^{2j-1}\epsilon_3-1$, and $e_3 = \gamma_{2j-2}\epsilon_1-1$. Again, we check that the simplex $\pb{1, \epsilon_1, \epsilon_1^{2j-1}\epsilon_3, \gamma_{2j-2}\epsilon_1}$ corresponding to this basis has integral volume $1$. With respect to this basis, the lattice points on $B_j^+$ are taken to the following coordinates:
\begin{align*}
1&\mapsto(0,0,0)=:\texttt{B1}, & \epsilon_1&\mapsto(1,0,0)=:\texttt{B2},\\
\epsilon_1^{2j-1}\epsilon_3&\mapsto(0,1,0)=:\texttt{B3}, & \epsilon_1^{2j}\epsilon_3&\mapsto(1,4,-5)=:\texttt{B4},\\
\gamma_{2j-2}\epsilon_1&\mapsto(0,0,1)=:\texttt{B5}, & \gamma_{2j-2}\epsilon_1^2&\mapsto(1,1,-1)=:\texttt{B6},\\
\gamma_{2j}&\mapsto(0,1,-1)=:\texttt{B7}, & \gamma_{2j}\epsilon_1&\mapsto(1,3,-4)=:\texttt{B8}.
\end{align*}
With the new coordinates it is straightforward to verify that $\IV(B_j^+) = 9$. Now we give a unimodular triangulation of $B$ into $9$ integer simplices with unit integer volume:
\begin{align*}
&\pb{\texttt{B1}, \texttt{B2}, \texttt{B5}, \texttt{B6}}, & &\pb{\texttt{B1}, \texttt{B2}, \texttt{B6}, \texttt{B8}}, & &\pb{\texttt{B1}, \texttt{B5}, \texttt{B6}, \texttt{B7}},\\
&\pb{\texttt{B1}, \texttt{B6}, \texttt{B7}, \texttt{B8}}, & &\pb{\texttt{B3}, \texttt{B4}, \texttt{B5}, \texttt{B6}}, & &\pb{\texttt{B3}, \texttt{B4}, \texttt{B5}, \texttt{B7}},\\
&\pb{\texttt{B4}, \texttt{B5}, \texttt{B6}, \texttt{B8}}, & &\pb{\texttt{B4}, \texttt{B5}, \texttt{B7}, \texttt{B8}}, & &\pb{\texttt{B5}, \texttt{B6}, \texttt{B7}, \texttt{B8}}.
\end{align*}

For the polytopes $B_j^-$, we take $1$ as the origin, and pick basis vectors $e_1 = \epsilon_1-1$, $e_2 = \epsilon_1^{2j-1}\epsilon_3^{-1}-1$, and $e_3 = \gamma_{2j-1}-1$. Again, we check that the simplex $\pb{1, \epsilon_1, \epsilon_1^{2j-1}\epsilon_3^{-1}, \gamma_{2j-1}}$ corresponding to this basis has integral volume $1$. With respect to this basis, the lattice points on $B_j^-$ are taken to the following coordinates:
\begin{align*}
1&\mapsto(0,0,0)=\texttt{B1}, & \epsilon_1&\mapsto(1,0,0)=\texttt{B2},\\
\epsilon_1^{2j-1}\epsilon_3^{-1}&\mapsto(0,1,0)=\texttt{B3}, & \epsilon_1^{2j}\epsilon_3^{-1}&\mapsto(1,4,-5)=\texttt{B4},\\
\gamma_{2j-1}&\mapsto(0,0,1)=\texttt{B5}, & \gamma_{2j-1}\epsilon_1&\mapsto(1,1,-1)=\texttt{B6},\\
\gamma_{2j}\epsilon_1^{-1}&\mapsto(0,1,-1)=\texttt{B7}, & \gamma_{2j}&\mapsto(1,3,-4)=\texttt{B8}.
\end{align*}
We observe that $B_j^+$ and $B_j^-$ are taken to the same polytope. So we also have $\IV(B_j^-) = 9$ and a unimodular triangulation of $B_j^-$. 

Finally, for simplices $C_j^+$ and $C_j^-$, we only have to verify that they have integral volume $1$, which is straightforward.
\end{proof}

\begin{prp}\label{prp:pfs} 
The sail $\mc S_K$ of the biquadratic field $K = K_n$ is given by
\[
\mc S_K = \bigcup_{m\in\Z^3} \epsilon_1^{m_1}\epsilon_2^{m_2}\epsilon_3^{m_3} \rb{ A \cup \bigcup_{j=1}^{3n} B_j^+ \cup \bigcup_{j=1}^{3n} B_j^- \cup \bigcup_{j=2}^{6n+1} C_j^+ \cup \bigcup_{j=1}^{6n+1} C_j^-}.
\]
\end{prp}
\begin{proof}
By \Cref{prp:t1}, each of the polytopes $A, B_j^{\pm}, C_j^{\pm}$ are contained in the some hyperplane $\Tr_{K/\Q}(\delta(-)) = 1$ with $\delta \in \mc O_K^{\vee,+}$. It follows from \Cref{lem:pis} that these polytopes are contained in the sail $\mc S_K$. It remains to check that there are no missing pieces. This can be proved by showing that the union of these polytopes have no boundary.

To do this, we match the faces of a polytope to another. With the help of the characterisation of the polytopes in the proof of \Cref{lem:pc}, we list the faces of the polytopes by lattice point elements they contain. We shall write the matching face's label in square brackets.

Here are the faces of the polytope $A$ (and their matching faces):
\begin{enumerate}[label=($A$-\Roman*)]
\item $\pbs{1, \epsilon_1, \epsilon_1^{6n+1}\epsilon_2}$, \quad [\textrm{($C_{6n+1}^-$-II)}]
\item $\pbs{\epsilon_1, \epsilon_1^{6n+1}\epsilon_2, \epsilon_1^{6n+2}\epsilon_2}$, \quad [$\epsilon_1 \textrm{($C_{6n+1}^-$-III)}$]
\item $\pbs{\epsilon_1^{6n+1}\epsilon_3, \epsilon_2\epsilon_3, \epsilon_1\epsilon_2\epsilon_3}$, \quad [$\epsilon_1^{6n+1}\epsilon_3 \textrm{($C_{6n+1}^+$-IV)}$]
\item $\pbs{\epsilon_1^{6n+1}\epsilon_3, \epsilon_1^{6n+2}\epsilon_3, \epsilon_1\epsilon_2\epsilon_3}$, \quad [$\epsilon_1^{6n+2}\epsilon_3 \textrm{($C_{6n+1}^+$-I)}$]
\item $\pbs{1, \epsilon_1, \gamma_{6n}\epsilon_1}$, \quad [\textrm{($B_{3n}^+$-II)}]
\item $\pbs{\epsilon_1^{6n+1}\epsilon_3, \epsilon_1^{6n+2}\epsilon_3, \gamma_{6n}\epsilon_1^2}$, \quad [$\epsilon_1^2 \textrm{($B_{3n}^+$-IV)}$]
\item $\pbs{\epsilon_2\epsilon_3, \epsilon_1\epsilon_2\epsilon_3, \gamma_{6n+1}\epsilon_2\epsilon_3}$, \quad [$\epsilon_2\epsilon_3 \textrm{($B_{3n}^-$-II)}$]
\item $\pbs{\epsilon_1^{6n+1}\epsilon_2, \epsilon_1^{6n+2}\epsilon_2, \gamma_{6n+1}\epsilon_1\epsilon_2\epsilon_3}$, \quad [$\epsilon_1^2\epsilon_2\epsilon_3 \textrm{($B_{3n}^-$-IV)}$]
\item $\pbs{\epsilon_1, \epsilon_1^{6n+1}\epsilon_3, \gamma_{6n}\epsilon_1, \gamma_{6n}\epsilon_1^2}$, \quad [$\epsilon_1 \textrm{($B_{3n}^+$-V)}$]
\item $\pbs{\epsilon_1^{6n+1}\epsilon_2, \epsilon_1\epsilon_2\epsilon_3, \gamma_{6n+1}\epsilon_2\epsilon_3, \gamma_{6n+1}\epsilon_1\epsilon_2\epsilon_3}$, \quad [$\epsilon_1\epsilon_2\epsilon_3 \textrm{($B_{3n}^-$-V)}$]
\item $\pbs{1, \epsilon_1^{6n+1}\epsilon_2, \epsilon_1^{6n+1}\epsilon_3, \epsilon_2\epsilon_3, \rho\epsilon_2, \gamma_{6n}\epsilon_1, \gamma_{6n+1}\epsilon_2\epsilon_3}$, \quad [$\epsilon_1^{-1} \textrm{($A$-XII)}$]
\item $\pbs{\epsilon_1, \epsilon_1^{6n+2}\epsilon_2, \epsilon_1^{6n+2}\epsilon_3, \epsilon_1\epsilon_2\epsilon_3, \rho\epsilon_1\epsilon_2, \gamma_{6n}\epsilon_1^2, \gamma_{6n+1}\epsilon_1\epsilon_2\epsilon_3}$.  \quad [$\epsilon_1 \textrm{($A$-XI)}$]
\end{enumerate}
Here are the faces of the polytopes $B_j^+$ (and their matching faces):
\begin{enumerate}[label=($B_j^+$-\Roman*)]
\item $\pbs{1, \epsilon_1, \gamma_{2j-2}\epsilon_1}$, \quad [\textrm{($B_{j-1}^+$-II)} if $j\ge 2$, $\epsilon_1^{-1}\epsilon_3 \textrm{($B_1^-$-III)}$ if $j=1$]
\item $\pbs{1, \epsilon_1, \gamma_{2j}\epsilon_1}$, \quad [\textrm{($B_{j+1}^+$-I)} if $j\le 3n-1$, \textrm{($A$-V)} if $j=3n$]
\item $\pbs{\epsilon_1^{2j-1}\epsilon_3, \epsilon_1^{2j}\epsilon_3, \gamma_{2j-2}\epsilon_1^2}$, \quad [$\epsilon_1^2 \textrm{($B_{j-1}^+$-IV)}$ if $j\ge 2$, $\epsilon_1\epsilon_3 \textrm{($B_1^-$-I)}$ if $j=1$]
\item $\pbs{\epsilon_1^{2j-1}\epsilon_3, \epsilon_1^{2j}\epsilon_3, \gamma_{2j}}$, \quad [$\epsilon_1^{-2} \textrm{($B_{j+1}^+$-III)}$ if $j\le 3n-1$, $\epsilon_1^{-2} \textrm{($A$-VI)}$ if $j=3n$]
\item $\pbs{1, \epsilon_1^{2j}\epsilon_3, \gamma_{2j}, \gamma_{2j}\epsilon_1}$, \quad [$\epsilon_1^{-1} \textrm{($B_{j+1}^+$-VI)}$ if $j\le 3n-1$, $\epsilon_1^{-1} \textrm{($A$-IX)}$ if $j=3n$]
\item $\pbs{\epsilon_1, \epsilon_1^{2j-1}\epsilon_3, \gamma_{2j-2}\epsilon_1, \gamma_{2j-2}\epsilon_1^2}$, \quad [$\epsilon_1 \textrm{($B_{j-1}^+$-V)}$ if $j\ge 2$, $\epsilon_3 \textrm{($B_1^-$-VI)}$ if $j=1$]
\item $\pbs{1, \epsilon_1^{2j-1}\epsilon_3, \gamma_{2j-2}\epsilon_1, \gamma_{2j}}$, \quad [$\epsilon_1^{-1} \textrm{($B_j^+$-VIII)}$]
\item $\pbs{\epsilon_1, \epsilon_1^{2j}\epsilon_3, \gamma_{2j-2}\epsilon_1^2, \gamma_{2j}\epsilon_1}$. \quad [$\epsilon_1 \textrm{($B_j^+$-VII)}$]
\end{enumerate}
Here are the faces of the polytopes $B_j^-$ (and their matching faces):
\begin{enumerate}[label=($B_j^-$-\Roman*)]
\item $\pbs{1, \epsilon_1, \gamma_{2j-1}}$, \quad [\textrm{($B_{j-1}^-$-II)} if $j\ge 2$, $\epsilon_1^{-1}\epsilon_3^{-1} \textrm{($B_1^+$-III)}$ if $j=1$]
\item $\pbs{1, \epsilon_1, \gamma_{2j+1}}$, \quad [\textrm{($B_{j+1}^-$-I)} if $j\le 3n-1$, $\epsilon_2^{-1}\epsilon_3^{-1} \textrm{($A$-VII)}$ if $j=3n$]
\item $\pbs{\epsilon_1^{2j-1}\epsilon_3^{-1}, \epsilon_1^{2j}\epsilon_3^{-1}, \gamma_{2j-1}\epsilon_1}$, \quad [$\epsilon_1^2 \textrm{($B_{j-1}^-$-IV)}$ if $j\ge 2$, $\epsilon_1\epsilon_3^{-1} \textrm{($B_1^+$-I)}$ if $j=1$]
\item $\pbs{\epsilon_1^{2j-1}\epsilon_3^{-1}, \epsilon_1^{2j}\epsilon_3^{-1}, \gamma_{2j+1}\epsilon_1^{-1}}$, \quad [$\epsilon_1^{-2} \textrm{($B_{j+1}^-$-III)}$ if $j\le 3n-1$, $\epsilon_1^{-2}\epsilon_2^{-1}\epsilon_3^{-1} \textrm{($A$-VIII)}$ if $j=3n$]
\item $\pbs{1, \epsilon_1^{2j}\epsilon_3^{-1}, \gamma_{2j+1}\epsilon_1^{-1}, \gamma_{2j+1}}$, \quad [$\epsilon_1^{-1} \textrm{($B_{j+1}^-$-VI)}$ if $j\le 3n-1$, $\epsilon_1^{-1}\epsilon_2^{-1}\epsilon_3^{-1} \textrm{($A$-X)}$ if $j=3n$]
\item $\pbs{\epsilon_1, \epsilon_1^{2j-1}\epsilon_3^{-1}, \gamma_{2j-1}, \gamma_{2j-1}\epsilon_1}$, \quad [$\epsilon_1 \textrm{($B_{j-1}^-$-V)}$ if $j\ge 2$, $\epsilon_3^{-1} \textrm{($B_1^+$-VI)}$ if $j=1$]
\item $\pbs{1, \epsilon_1^{2j-1}\epsilon_3^{-1}, \gamma_{2j-1}, \gamma_{2j+1}\epsilon_1^{-1}}$, \quad [$\epsilon_1^{-1} \textrm{($B_j^-$-VIII)}$]
\item $\pbs{\epsilon_1, \epsilon_1^{2j}\epsilon_3^{-1}, \gamma_{2j-1}\epsilon_1, \gamma_{2j+1}}$. \quad [$\epsilon_1 \textrm{($B_j^-$-VII)}$]
\end{enumerate}
Here are the faces of the polytopes $C_j^+$ (and their matching faces):
\begin{enumerate}[label=($C_j^+$-\Roman*)]
\item $\pbs{\epsilon_1^{-1}, 1, \epsilon_1^{-j}\epsilon_2}$, \quad [\textrm{($C_{j+1}^+$-II)} if $j\le 6n$, $\epsilon_1^{-6n-2}\epsilon_3^{-1} \textrm{($A$-IV)}$ if $j=6n+1$]
\item $\pbs{\epsilon_1^{-1}, 1, \epsilon_1^{-j+1}\epsilon_2}$, \quad [\textrm{($C_{j-1}^+$-I)} if $j\ge 3$, $\epsilon_1^{-1} \textrm{($C_1^-$-I)}$ if $j=2$]
\item $\pbs{\epsilon_1^{-1}, \epsilon_1^{-j}\epsilon_2, \epsilon_1^{-j+1}\epsilon_2}$, \quad [$\epsilon_1^{-1} \textrm{($C_{j-1}^+$-IV)}$ if $j\ge 3$, $\epsilon_1^{-2} \textrm{($C_1^-$-IV)}$ if $j=2$]
\item $\pbs{1, \epsilon_1^{-j}\epsilon_2, \epsilon_1^{-j+1}\epsilon_2}$. \quad [$\epsilon_1 \textrm{($C_{j+1}^+$-III)}$ if $j\le 6n$, $\epsilon_1^{-6n-1}\epsilon_3^{-1} \textrm{($A$-III)}$ if $j=6n+1$]
\end{enumerate}
Here are the faces of the polytopes $C_j^-$ (and their matching faces):
\begin{enumerate}[label=($C_j^-$-\Roman*)]
\item $\pbs{1, \epsilon_1, \epsilon_1^{j-1}\epsilon_2}$, \quad [\textrm{($C_{j-1}^-$-II)} if $j\ge 2$, $\epsilon_1 \textrm{($C_2^+$-II)}$ if $j=1$]
\item $\pbs{1, \epsilon_1, \epsilon_1^j\epsilon_2}$, \quad [\textrm{($C_{j+1}^-$-I)} if $j\le 6n$, \textrm{($A$-I)} if $j=6n+1$]
\item $\pbs{1, \epsilon_1^{j-1}\epsilon_2, \epsilon_1^j\epsilon_2}$, \quad [$\epsilon_1^{-1} \textrm{($C_{j+1}^-$-IV)}$ if $j\le 6n$, $\epsilon_1^{-1} \textrm{($A$-II)}$ if $j=6n+1$]
\item $\pbs{\epsilon_1, \epsilon_1^{j-1}\epsilon_2, \epsilon_1^j\epsilon_2}$. \quad [$\epsilon_1 \textrm{($C_{j-1}^-$-III)}$ if $j\ge 2$, $\epsilon_1^2 \textrm{($C_2^+$-III)}$ if $j=1$]
\end{enumerate}
Since all faces are matched, the proposition is established.
\end{proof}

Now we are ready to prove \Cref{thm:pfi}.

\begin{proof}[Proof of Theorem $\ref{thm:pfi}$]
From \Cref{prp:pfi}, we see that the elements $1,\rho,\gamma_j$ are indeed indecomposable. On the other hand, \Cref{prp:pfs}, the sail $\mc S_K$ of $K$ is given by the union of the polytopes $A$, $B_j^\pm$, $C_j^\pm$ and their translates by totally positive units. Since the polytopes $A$, $B_j^\pm$, $C_j^\pm$ all have integer distance $1$ and admits unimodular triangulations, we use \Cref{thm:ii} and conclude that all the indecomposable elements lie on the sail $\mc S_K$. Finally, using the characterisation of these polytopes in the proof of \Cref{lem:pc}, we see that the polytopes $A$, $B_j^\pm$, $C_j^\pm$ contains only translates of $1,\rho,\gamma_j$ by totally positive units, and no other lattice points. So the list of indecomposable elements is complete.

Concerning the growth of $\iota(K)$ with respect to the discriminant, we use \eqref{eq:pd} and the fact that $x_n$ is well approximated by $(\frac{1+\sqrt{5}}2)^n$. Recalling that $w = \log (\frac{1+\sqrt{5}}2)$, it is easily verified that for $n\in\N_0$ we have
\[
(24n+6)w - \frac 15 \le \log (x_{24n+6}-3) \le (24n+6)w.
\]
It follows that 
\[
(48n+12)w - \frac 15 + \log 16 \le \log \Delta_K \le (48n+12)w + \log 16.
\]
A rough approximation of the constants above then gives a bound
\[
\frac{\log \Delta_K}{8w} \le \iota(K) \le \frac{\log \Delta_K}{8w} + 1. \qedhere
\]
\end{proof}

\newcommand{\etalchar}[1]{$^{#1}$}


\begin{thebibliography}{{\v{C}}LS{\etalchar{+}}19}

\bibitem[BH11]{BH}
M.~Bhargava and J.~Hanke.
\newblock Universal quadratic forms and the 290-theorem, 2011.
\newblock {P}reprint.

\bibitem[BK15]{BK1}
V.~Blomer and V.~Kala.
\newblock Number fields without universal $n$-ary quadratic forms.
\newblock {\em Math. Proc. Cambridge Philos. Soc.}, 159:239--252, 2015.

\bibitem[BK18]{BK2}
V.~Blomer and V.~Kala.
\newblock On the rank of universal quadratic forms over real quadratic fields.
\newblock {\em Doc. Math.}, 23:15--34, 2018.

\bibitem[Cha73]{Chatelain1973}
D.~Chatelain.
\newblock Bases des entiers des corps compos{\'e}s par des extensions
  quadratiques de $\mathbb{Q}$.
\newblock {\em Ann. Sci. Univ. Besan{\c c}on Math.}, 6:38, 1973.

\bibitem[CKR96]{CKR}
W.~K. Chan, M.-H. Kim, and S.~Raghavan.
\newblock Ternary universal integral quadratic forms.
\newblock {\em Japan. J. Math.}, 22:263--273, 1996.

\bibitem[{\v{C}}LS{\etalchar{+}}19]{CL+}
M.~{\v{C}}ech, D.~Lachman, J.~Svoboda, M.~Tinkov{\'a}, and K.~Zemkov{\'a}.
\newblock Universal quadratic forms and indecomposables over biquadratic
  fields.
\newblock {\em Math. Nachr.}, 292:540--555, 2019.

\bibitem[Cus84]{Cusick1984}
T.~Cusick.
\newblock Lower bounds for regulators.
\newblock In {\em Number Theory {N}oordwijkerhout}, volume 1068 of {\em Lecture
  Notes in Mathematics}, pages 63--73. Springer Berlin-Heidelberg, 1984.

\bibitem[DDK19]{DDK2019}
D.~S. Dummit, E.~P. Dummit, and H.~Kisilevsky.
\newblock Signature ranks of units in cyclotomic extensions of abelian number
  fields.
\newblock {\em Pacific J. Math.}, 298(2):285--298, 2019.

\bibitem[DS82]{DS}
A.~Dress and R.~Scharlau.
\newblock Indecomposable totally positive numbers in real quadratic orders.
\newblock {\em J. Number Theory}, 14:292--306, 1982.

\bibitem[EH82]{EH2}
D.~R. Estes and J.~S. Hsia.
\newblock Exceptional integers of some ternary quadratic forms.
\newblock {\em Adv. Math.}, 45(3):310--318, 1982.

\bibitem[EK97]{EK1}
A.~G. Earnest and A.~Khosravani.
\newblock Universal positive quaternary quadratic lattices over totally real
  number fields.
\newblock {\em Mathematika}, 44:342--347, 1997.

\bibitem[HHX23]{HHX}
Z.~He, Y.~Hu, and F.~Xu.
\newblock On indefinite $k$-universal integral quadratic forms over number
  fields.
\newblock {\em Math. Z.}, 304:20, 2023.

\bibitem[HJ68]{jacobi}
E.~Heine and C.~G.~J. Jacobi.
\newblock Allgemeine {T}heorie der kettenbruch{\"a}hnlichen {A}lgorithmen, in
  welchen jede {Z}ahl aus drei vorhergehenden gebildet wird.
\newblock {\em J. Reine Angew. Math.}, 69:29--64, 1868.

\bibitem[HKK78]{HKK1978}
J.~S. Hsia, Y.~Kitaoka, and M.~Kneser.
\newblock Representations of positive definite quadratic forms.
\newblock {\em J. Reine Angew. Math.}, 301:132--141, 1978.

\bibitem[JW09]{JW2009}
M.~J. Jacobson and H.~C. Williams.
\newblock {\em Solving the {P}ell equation}.
\newblock CMS Books in Mathematics. Springer New York, 2009.

\bibitem[Kal16]{Ka1}
V.~Kala.
\newblock Universal quadratic forms and elements of small norm in real
  quadratic fields.
\newblock {\em Bull. Aust. Math. Soc.}, 94:7--14, 2016.

\bibitem[Kal23]{Ka4}
V.~Kala.
\newblock Number fields without universal quadratic forms of small rank exist
  in most degrees.
\newblock {\em Math. Proc. Cambridge Philos. Soc.}, 174:225--231, 2023.

\bibitem[Kar22]{Karpenkov2022}
O.~Karpenkov.
\newblock {\em Geometry of continued fractions}, volume~26 of {\em Algorithms
  and Computation in Mathematics}.
\newblock Springer Berlin, Heidelberg, 2nd edition, 2022.

\bibitem[Kar24]{Karp}
O.~Karpenkov.
\newblock On a periodic {Jacobi-Perron} type algorithm.
\newblock {\em Monatsh. Math.}, 205:531--601, 2024.

\bibitem[Khi64]{Khinchin1964}
A.~Khinchin.
\newblock {\em Continued fractions}.
\newblock Dover Publications, 1964.

\bibitem[Kim00]{Ki2}
B.~M. Kim.
\newblock Universal octonary diagonal forms over some real quadratic fields.
\newblock {\em Comment. Math. Helv.}, 75:410--414, 2000.

\bibitem[KKP22]{KKP}
B.~M. Kim, M.-H. Kim, and D.~Park.
\newblock Real quadratic fields admitting universal lattices of rank 7.
\newblock {\em J. Number Theory}, 233:456--466, 2022.

\bibitem[KL24]{KL}
D.~Kim and S.~H. Lee.
\newblock Lifting problem for universal quadratic forms over totally real cubic
  number fields.
\newblock {\em Bull. Lond. Math. Soc.}, 56(3):1192--1206, 2024.

\bibitem[Kle95]{Klein}
F.~Klein.
\newblock {Ueber eine geometrische Auffassung der gew\" ohnlichen
  Kettenbruchentwicklung}.
\newblock {\em Nachr. Ges. Wiss. G{\"o}ttingen}, pages 357--359, 1895.

\bibitem[KST23]{KST}
V.~Kala, E.~Sgallov{\'a}, and M.~Tinkov{\'a}.
\newblock Arithmetic of cubic number fields: {Jacobi--Perron, Pythagoras,} and
  indecomposables. To appear in J. Number Theory, 2023.
\newblock arxiv:2303.00485.

\bibitem[KT23]{KT2023}
V.~Kala and M.~Tinkov{\'a}.
\newblock Universal quadratic forms, small norms, and traces in families of
  number fields.
\newblock {\em Int. Math. Res. Not.}, 2023(9):7541--7577, 2023.

\bibitem[KTZ20]{KTZ2020}
J.~Kr{\'a}sensk{\'y}, M.~Tinkov{\'a}, and K.~Zemkov{\'a}.
\newblock There are no universal ternary quadratic forms over biquadratic
  fields.
\newblock {\em Proc. Edinb. Math. Soc. (2)}, 63:861--912, 2020.

\bibitem[Kub56]{Kubota1956}
T.~Kubota.
\newblock {\"U}ber den bizyklischen biquadratischen {Z}ahlk{\"o}rper.
\newblock {\em Nagoya Math J.}, 10:65--85, 1956.

\bibitem[KY21]{KY2021}
V.~Kala and P.~Yatsyna.
\newblock Lifting problem for universal quadratic forms.
\newblock {\em Adv. Math.}, 377:107497, 2021.

\bibitem[KY23]{KY3}
V.~Kala and P.~Yatsyna.
\newblock On {K}itaoka's conjecture and lifting problem for universal quadratic
  forms.
\newblock {\em Bull. Lond. Math. Soc.}, 55:854--864, 2023.

\bibitem[KY{\.Z}23]{KYZp2023}
V.~Kala, P.~Yatsyna, and B.~{\.Z}mija.
\newblock Real quadratic fields with a universal form of given rank have
  density zero, 2023.
\newblock arxiv:2302.12080.

\bibitem[Maa41]{Ma}
H.~Maa{\ss}.
\newblock {\"U}ber die {D}arstellung total positiver {Z}ahlen des {K}\" orpers
  {$R(\sqrt 5)$} als {S}umme von drei {Q}uadraten.
\newblock {\em Abh. Math. Sem. Univ. Hamburg}, 14:185--191, 1941.

\bibitem[Man24]{Man}
S.~H. Man.
\newblock Minimal rank of universal lattices and number of indecomposable
  elements in real multiquadratic fields.
\newblock {\em Adv. Math.}, 447:109694, 2024.

\bibitem[Per07]{Perron}
O.~Perron.
\newblock Grundlagen f{\"u}r eine {T}heorie des {J}acobischen
  {K}ettenalgorithmus.
\newblock {\em Math. Ann.}, 64:1--76, 1907.

\bibitem[Per77]{Perron1977}
O.~Perron.
\newblock {\em Die {L}ehre von den {K}ettenbr{\"u}chen, {B}and I: Elementare
  {K}ettenbr{\"u}che}.
\newblock Vieweg+Teubner Verlag Wiesbaden, 3rd edition, 1977.

\bibitem[RSD23]{RDS}
O.~Regev and N.~Stephens-Davidowitz.
\newblock A simple proof of a reverse Minkowski theorem for integral lattices,
  2023.
\newblock arxiv:2306.03697.

\bibitem[Sch89]{Schmal1989}
B.~Schmal.
\newblock Diskriminanten, $\mathbb{Z}$-{G}anzheitsbasen und relative
  {G}anzheitsbasen bei multiquadratischen {Z}ahlk{\"o}rpern.
\newblock {\em Arch. Math.}, 52:245--257, 1989.

\bibitem[Sha74]{Shanks1974}
D.~Shanks.
\newblock The simplest cubic fields.
\newblock {\em Math. Comp.}, 28(128):1137--1152, 1974.

\bibitem[Shi76]{Shintani1976}
T.~Shintani.
\newblock On evaluation of zeta functions of totally real algebraic number
  fields at non-positive integers.
\newblock {\em J. Fac. Sci. Univ. Tokyo Sect. IA Math.}, 23(2):393--417, 1976.

\bibitem[Sie45]{Si2}
C.~L. Siegel.
\newblock Sums of $m$-th powers of algebraic integers.
\newblock {\em Ann. of Math.}, 46:313--339, 1945.

\bibitem[Tin23]{Ti1}
M.~Tinkov{\'a}.
\newblock Trace and norm of indecomposable integers in cubic orders.
\newblock {\em Ramanujan J.}, 61(4):1121--1144, 2023.

\bibitem[Vou16]{Voutier}
P.~Voutier.
\newblock Families of periodic {Jacobi--Perron} algorithms for all period
  lengths.
\newblock {\em J. Number Theory}, 168:472--486, 2016.

\bibitem[Wil70]{Williams1970}
K.~S. Williams.
\newblock Integers of biquadratic fields.
\newblock {\em Canad. Math. Bull.}, 13:519--526, 1970.

\bibitem[Yat19]{Ya}
P.~Yatsyna.
\newblock A lower bound for the rank of a universal quadratic form with integer
  coefficients in a totally real field.
\newblock {\em Comment. Math. Helvet.}, 94:221--239, 2019.

\end{thebibliography}
\end{document}